\documentclass[11pt,reqno]{amsart}

\usepackage{amsmath,amssymb,mathrsfs}
\usepackage{graphicx,cite,cases,times}
\newtheorem{theo}{Theorem}[section]

\newtheorem{lemm}[theo]{Lemma}
\newtheorem{prop}[theo]{Proposition}
\newtheorem{rema}[theo]{Remark}

\numberwithin{equation}{section}

\begin{document}

\title[elastic wave scattering by biperiodic structures]{Convergence of the PML
solution for elastic wave scattering by biperiodic structures}

\author{Xue Jiang}
\address{School of Science, Beijing University of Posts and Telecommunications,
Beijing 100876, China.}
\email{jxue@lsec.cc.ac.cn}

\author{Peijun Li}
\address{Department of Mathematics, Purdue University, West Lafayette, IN 47907,
USA.}
\email{lipeijun@math.purdue.edu}

\author{Junliang Lv}
\address{School of Mathematics, Jilin University, Changchun 130012, China. }
\email{lvjl@jlu.edu.cn}

\author{Weiying Zheng}
\address{LSEC, NCMIS, Academy of Mathematics and System
Sciences, CAS; University of Chinese Academy of Sciences, Beijing,
100190, China.}
\email{zwy@lsec.cc.ac.cn}

\thanks{The research of XJ was supported in part by China NSF grant 11401040 and
by the Fundamental Research Funds for the Central Universities 24820152015RC17.
The research of PL was supported in part by the NSF grant DMS-1151308. The
research of JL was partially supported by the China NSF grants 11126040 and
11301214. The author of WZ was supported in part by China NSF 91430215 and by
the National Magnetic Confinement Fusion Science Program (2015GB110003).}

\subjclass[2010]{65N30, 78A45, 35Q60}

\keywords{Elastic wave equation, biperiodic structure, perfectly matched
layer, transparent boundary condition}

\begin{abstract}
This paper is concerned with the analysis of elastic wave scattering of a
time-harmonic plane wave by a biperiodic rigid surface, where the wave
propagation is governed by the three-dimensional Navier equation. An exact
transparent boundary condition is developed to reduce the scattering problem
equivalently into a boundary value problem in a bounded domain. The perfectly
matched layer (PML) technique is adopted to truncate the unbounded physical
domain into a bounded computational domain. The well-posedness and exponential
convergence of the solution are established for the truncated PML problem by
developing a PML equivalent transparent boundary condition. The proofs rely on
a careful study of the error between the two transparent boundary operators. The
work significantly extend the results from the one-dimensional periodic
structures to the two-dimensional biperiodic structures. Numerical experiments
are included to demonstrate the competitive behavior of the proposed method.
\end{abstract}

\maketitle

\section{Introduction}

Scattering theory in periodic structures has many important applications in
diffractive optics \cite{bcm-01, bdc-josa95}, where the periodic structures are
often named as gratings. The scattering problems have been studied extensively
in periodic structures by many researchers for all the commonly encountered
waves including the acoustic, electromagnetic, and elastic waves 
\cite{a-jiea99, a-mmas99, b-sjna95, b-sjam97, cf-tams91, df-jmaa92, eh-mmas10,
eh-mmmas12, lwz-ip15, wbllw-sjna15}. The governing equations of these waves are
known as the Helmholtz equation, the Maxwell equations, and the Navier equation,
respectively. In this paper, we consider the scattering of a time-harmonic
elastic plane wave by a biperiodic rigid surface, which is also called a
two-dimensional or crossed grating. The elastic wave scattering problems have
received ever-increasing attention in both engineering and mathematical
communities for their important applications in geophysics and seismology. The
elastic wave motion is governed by the three-dimensional Navier equation. A
fundamental challenge of this problem is to truncate the unbounded physical
domain into a bounded computational domain. An appropriate boundary condition is
needed on the boundary of the truncated domain to avoid artificial wave
reflection. We adopt the perfectly matched layer (PML) technique to handle this
issue.

The research on the PML technique has undergone a tremendous development since
B\'{e}renger proposed a PML for solving the time-dependent Maxwell equations
\cite{b-jcp94}. The basis idea of the PML technique is to surround the domain of
interest by a layer of finite thickness fictitious material which absorbs all
the waves coming from inside the computational domain. When the waves
reach the outer boundary of the PML region, their values are so small that
the homogeneous Dirichlet boundary conditions can be imposed. Various
constructions of PML absorbing layers have been proposed and investigated
for the acoustic and electromagnetic wave scattering problems \cite{bw-sjna05,
bp-mc07, cm-sjsc98, ct-g01, cw-motl94, hsz-sjma03, ls-c98, ty-anm98}. In
particular, combined with the PML technique, an effective adaptive finite
element method was proposed in \cite{bcw-josa05, cw-sjna03} to solve the
two-dimensional diffraction grating problem where the one-dimensional grating
structure was considered. Due to the competitive numerical performance, the
method was quickly adopted to solve many other scattering problems including the
obstacle scattering problems \cite{cl-sjna05, cc-mc08} and the three-dimensional
diffraction grating problem \cite{blw-mc10}. However, the PML technique is much
less studied for the elastic wave scattering problems \cite{hsb-jasa96},
especially for the rigorous convergence analysis. We refer to \cite{bpt-mc10,
cxz-mc} for recent study on convergence analysis of the elastic obstacle
scattering problems. 

Recently, we have proposed an adaptive finite element method combining
with the PML technique to solve the elastic scattering problem in
one-dimensional periodic structures \cite{jllz}. Using the quasi-periodicity of
the solution, we develop a transparent boundary condition and formulate the
scattering problem equivalently into a boundary value problem in a bounded
domain. Following the complex coordinate stretching, we study the truncated PML
problem and show that it has a unique weak solution which converges
exponentially to the solution of the original scattering problem. 

The purpose of this paper is to extend our previous work on the one-dimensional
periodic structures in \cite{jllz} to the two-dimensional biperiodic
structures. We point out that the extension is nontrivial because the more
complicated three-dimensional Navier equation needs to be considered. The
analysis is mathematically more sophisticated and the numerics
is computationally more intense. This work presents an important application of
the PML method for the scattering problem of the elastic wave equation. The
elastic wave equation is complicated due to the coexistence of compressional and
shear waves that have different wavenumbers. To take into account this feature,
we introduce two potential functions, one scalar and one vector, to split the
wave field into its compressional and shear parts via the Helmholtz
decomposition. As a consequence, the scalar potential function satisfies the
Helmholtz equation while the vector potential function satisfies the Maxwell
equation. Using these two potential functions, we develop an exact transparent
boundary condition to reduce the scattering problem from an open domain into a
boundary value problem in a bounded domain. The energy conservation is proved
for the propagating wave modes of the model problem and is used for
verification of our numerical results. The well-posedness and exponential
convergence of the solution are established for the truncated PML problem by
developing a PML equivalent transparent boundary condition. The proofs rely on
a careful study of the error between the two transparent boundary operators.
Two numerical examples are also included to demonstrate the competitive behavior
of the proposed method.

The paper is organized as follows. In section 2, we introduce the model problem
of the elastic wave scattering by a biperiodic surface and formulate it into a
boundary value problem by using a transparent boundary condition. In section 3,
we introduce the PML formulation and prove the well-posedness and convergence of
the truncated PML problem. In section 4, we discuss the numerical implementation
of our numerical algorithm and present some numerical experiments to illustrate
the performance of the proposed method. The paper is concluded with some general
remarks in section 5.

\section{Problem formulation}

In this section, we introduce the model problem and present an exact transparent
boundary condition to reduce the problem into a boundary value problem in a
bounded domain. The energy distribution will be studied for the diffracted
propagating waves of the scattering problem.

\subsection{Navier equation}

Let $\boldsymbol r=(x_1, x_2)^\top$ and $\boldsymbol x=(x_1, x_2,
x_3)^\top$. Consider the elastic scattering of a time-harmonic plane wave by a
biperiodic surface $\Gamma_f=\{\boldsymbol x\in\mathbb R^3: x_3=f(\boldsymbol
r)\}$, where $f$ is a Lipschitz continuous and biperiodic function with period
$(\Lambda_1, \Lambda_2)$ in $(x_1, x_2)$. Denote by $\Omega_f=\{\boldsymbol
x\in\mathbb R^3: x_3>f(\boldsymbol r)\}$ the open space above $\Gamma_f$. Let
$h$ be a constant satisfying $h>\max_{\boldsymbol r\in\mathbb
R^2} f(\boldsymbol r)$. Denote $\Omega=\{\boldsymbol{x}\in\mathbb{R}^3: 0<x_1<\Lambda_1,\,
0<x_2<\Lambda_2,\, f(\boldsymbol r)<x_3<h\}$ and
$\Gamma_h=\{\boldsymbol{x}\in\mathbb{R}^3: 0<x_1<\Lambda_1,\, 0<x_2<\Lambda_2,\,
x_3=h\}$. Let $\Omega_h=\{\boldsymbol{x}\in\mathbb{R}^3: 0<x_1<\Lambda_1,\,
0<x_2<\Lambda_2,\, x_3>h\}$ be the open space above $\Gamma_h$.

The propagation of a time-harmonic elastic wave is governed by the
Navier equation:
\begin{equation}\label{une}
 \mu\Delta\boldsymbol{u}+(\lambda+\mu)\nabla\nabla\cdot\boldsymbol{u}
+\omega^2\boldsymbol{u}=0\quad\text{in} ~ \Omega_f,
\end{equation}
where $\boldsymbol{u}=(u_1, u_2,u_3)^\top$ is the displacement vector of the
total elastic wave field, $\omega>0$ is the angular frequency, $\mu$ and
$\lambda$ are the Lam\'{e} constants satisfying $\mu>0$ and $\lambda+\mu>0$.
Assuming that the surface $\Gamma_f$ is elastically rigid, we have
\begin{equation}\label{bc}
 \boldsymbol{u}=0\quad\hbox{on}~\Gamma_f.
\end{equation}
Define
\[
 \kappa_{1}=\frac{\omega}{(\lambda+2\mu)^{1/2}}\quad\text{and}\quad
\kappa_{2}=\frac{\omega}{\mu^{1/2}},
\]
which are known as the compressional wavenumber and the shear wavenumber,
respectively.

Let the scattering surface $\Gamma_f$ be illuminated from above by a
time-harmonic compressional plane wave:
\[
 \boldsymbol{u}^{\rm inc}(\boldsymbol x)=\boldsymbol{q}e^{{\rm
i}\kappa_{1}\boldsymbol{x}\cdot\boldsymbol{q}},
\]
where $\boldsymbol{q}=(\sin\theta_1\cos\theta_2,\,\sin\theta_1\sin\theta_2,\,
-\cos\theta_1)^\top$ is the propagation direction vector, and $\theta_1,
\theta_2$ are called the latitudinal and longitudinal incident angles satisfying
$\theta_1\in[0, \pi/2), \theta_2\in[0,2\pi]$. It can be verified that the
incident wave also satisfies the Navier equation:
\begin{equation}\label{uine}
\mu\Delta\boldsymbol{u}^{\rm inc}+(\lambda
+\mu)\nabla\nabla\cdot\boldsymbol{u}^{\rm inc}
+\omega^2\boldsymbol{u}^{\rm inc}=0\quad\text{in} ~ \Omega_f.
\end{equation}

\begin{rema}
 The scattering surface may be also illuminated by a time-harmonic shear plane
wave:
\[
 \boldsymbol{u}^{\rm inc}=\boldsymbol p e^{{\rm i}\kappa_2\boldsymbol
x\cdot\boldsymbol q},
\]
where $\boldsymbol p$ is the polarization vector satisfying $\boldsymbol
p\cdot\boldsymbol q=0$. More generally, the scattering surface can be
illuminated by any linear combination of the time-harmonic compressional and
shear plane waves. For clarity, we take the time-harmonic compressional plane
wave as an example since the results and analysis are the same for other forms
of the incident wave.
\end{rema}

Motivated by uniqueness, we are interested in a quasi-periodic solution of
$\boldsymbol{u}$, i.e., $\boldsymbol{u}(\boldsymbol{x})e^{-{\rm
i}\boldsymbol{\alpha}\cdot\boldsymbol{r}}$ is biperiodic in $x_1$ and $x_2$ with
periods $\Lambda_1$ and $\Lambda_2$, respectively.
Here $\boldsymbol{\alpha}=(\alpha_1,\alpha_2)^\top$
with $\alpha_1=\kappa_{1}\sin\theta_1\cos\theta_2,
\alpha_2=\kappa_{1}\sin\theta_1\sin\theta_2$. In addition, the
following radiation condition is imposed: the total displacement
$\boldsymbol{u}$ consists of bounded outgoing waves plus the incident wave in
$\Omega_h$.

We introduce some notation and Sobolev spaces. Let $\boldsymbol{u}=(u_1,
u_2,u_3)^\top$ be a vector function. Define the
Jacobian matrix of $\boldsymbol{u}$:
\[
 \nabla\boldsymbol{u}=\begin{bmatrix}
                       \partial_{x_1} u_1 & \partial_{x_2} u_1&\partial_{x_3} u_1\\
                       \partial_{x_1} u_2 & \partial_{x_2} u_2&\partial_{x_3} u_2\\
                       \partial_{x_1} u_3 & \partial_{x_2} u_3&\partial_{x_3} u_3
                      \end{bmatrix}.
\]
Define a quasi-biperiodic functional space
\begin{align*}
 H^1_{\rm qp}(\Omega)=\{u\in H^1(\Omega):
&u(x_1+n_1\Lambda_1,x_2+n_2\Lambda_2, x_3)\\
&=u(x_1,x_2,x_3)e^{{\rm i}(n_1\alpha_1\Lambda_1+n_2\alpha_2\Lambda_2)},\,
n=(n_1, n_2)^\top\in\mathbb{Z}^2\},
\end{align*}
which is a subspace of $H^1(\Omega)$ with the norm $\|\cdot\|_{H^1(\Omega)}$.
For any quasi-biperiodic function $u$ defined on $\Gamma_h$, it admits the
Fourier series expansion:
\[
 u(\boldsymbol{r}, h)=\sum_{n\in\mathbb{Z}^2}u^{(n)}(h)e^{{\rm
i}\boldsymbol{\alpha}^{(n)}\cdot\boldsymbol{r}},
\]
where $\boldsymbol{\alpha}^{(n)}=(\alpha^{(n)}_1, \alpha^{(n)}_2)^\top,
\alpha_1^{(n)}=\alpha_1+2\pi n_1/\Lambda_1, \alpha_2^{(n)}=\alpha_2+2\pi
n_2/\Lambda_2$, and
\[
u^{(n)}(h)=\frac{1}{\Lambda_1\Lambda_2}\int_0^{\Lambda_1}\int_0^{\Lambda_2}
u(\boldsymbol r,h)e^{-{\rm i}\boldsymbol{\alpha}^{(n)}\cdot\boldsymbol{r}}{\rm
d}\boldsymbol r.
\]
We define a trace functional space $H^s(\Gamma_h)$ with the norm
given by
\[
 \|u\|_{H^s(\Gamma_h)}=\Bigl(\Lambda_1\Lambda_2
\sum_{n\in\mathbb{Z}^2}(1+|\boldsymbol{\alpha}^{(n)}|^2)^s
|u^{(n)}(h)|^2\Bigr)^{1/2}.
\]
Let $H^1_{\rm qp}(\Omega)^3$ and $H^s(\Gamma_h)^3$ be the Cartesian product
spaces equipped with the corresponding 2-norms of $H^1_{\rm qp}(\Omega)$ and
$H^s(\Gamma_h)$, respectively. It is known that $H^{-s}(\Gamma_h)^3$ is
the dual space of $H^s(\Gamma_h)^3$ with respect to the
$L^2(\Gamma_h)^3$ inner product
\[
\langle \boldsymbol{u},
\boldsymbol{v}\rangle_{\Gamma_h}=\int_{\Gamma_h}\boldsymbol{u}\cdot
\bar{\boldsymbol{v}}\,{\rm d}\boldsymbol{r},
\]
where the bar denotes the complex conjugate.

\subsection{Boundary value problem}

We wish to reduce the problem equivalently into a boundary value problem in
$\Omega$ by introducing an exact transparent boundary condition on $\Gamma_h$.

The total field $\boldsymbol{u}$ consists of the incident field
$\boldsymbol{u}^{\rm inc}$ and the diffracted field
$\boldsymbol{v}$, i.e.,
\begin{equation}\label{tf}
 \boldsymbol{u}=\boldsymbol{u}^{\rm inc}+\boldsymbol{v}.
\end{equation}
Subtracting \eqref{uine} from \eqref{une} and noting \eqref{tf}, we obtain the
Navier equation for the diffracted field
$\boldsymbol{v}$:
\begin{equation}\label{vne}
\mu\Delta\boldsymbol{v}+(\lambda +\mu)\nabla\nabla\cdot\boldsymbol{v}
+\omega^2\boldsymbol{v}=0\quad\text{in} ~ \Omega_h.
\end{equation}
For any solution $\boldsymbol{v}$ of \eqref{vne}, we introduce the Helmholtz
decomposition to split it into the compressional and shear parts:
\begin{equation}\label{hdv}
 \boldsymbol{v}=\nabla\phi +\nabla\times\boldsymbol{\psi},
   \quad \nabla\cdot\boldsymbol{\psi}=0,
\end{equation}
where $\phi$ is a scalar potential function and $\boldsymbol{\psi}$
is a vector potential function. Substituting \eqref{hdv} into \eqref{vne} gives
\[
 (\lambda +2\mu)\nabla\left(\Delta\phi +\omega^2\phi
\right)+\mu\nabla\times(\Delta\boldsymbol{\psi}+\omega^2\boldsymbol{\psi})=0,
\]
which is fulfilled if $\phi$ and $\boldsymbol\psi$ satisfy the Helmholtz
equation:
\begin{equation}\label{he}
 \Delta\phi +\kappa_{1}^2\phi=0,\quad \Delta\boldsymbol{\psi}
+\kappa_{2}^2\boldsymbol{\psi}=0.
\end{equation}
It follows from $\nabla\cdot\boldsymbol\psi=0$ and \eqref{he} that the vector
potential function $\boldsymbol\psi$ satisfies the Maxwell equation:
\[
 \nabla\times(\nabla\times\boldsymbol\psi)-\kappa_2^2\boldsymbol\psi=0.
\]

Since $\boldsymbol{v}$ is a quasi-biperiodic function, we have from \eqref{hdv}
that $\phi$ and $\boldsymbol{\psi}=(\psi_{1},\psi_{2},\psi_{3})^\top$ are also
quasi-biperiodic functions. They have the Fourier series expansions:
\[
 \phi(\boldsymbol x)=\sum_{n\in\mathbb{Z}^2}\phi^{(n)}(x_3)e^{{\rm
i}\boldsymbol{\alpha}^{(n)}\cdot\boldsymbol{r}},\quad
 \boldsymbol{\psi}^{(n)}(\boldsymbol x)
 =\sum_{n\in\mathbb{Z}^2}\boldsymbol{\psi}^{(n)}(x_3)e^{{\rm
i}\boldsymbol{\alpha}^{(n)}\cdot\boldsymbol{r}}.
\]
Plugging the above Fourier series into \eqref{he} yields
\[
 \frac{{\rm d}^2\phi^{(n)}(x_3)}{{\rm
d}x_3^2}+\bigl(\beta^{(n)}_{1}\bigr)^2\phi^{(n)}(x_3)=0,\quad
 \frac{{\rm d}^2\boldsymbol{\psi}^{(n)}(x_3)}{{\rm
d}x_3^2}+\bigl(\beta^{(n)}_{2}\bigr)^2
\boldsymbol{\psi}^{(n)}(x_3)=0,
\]
where
\begin{equation}\label{beta}
\beta_{j}^{(n)}=
\begin{cases}
\big(\kappa^2_{j} - |\boldsymbol{\alpha}^{(n)}|^2\big)^{1/2},\quad
&|\boldsymbol{\alpha}^{(n)}|<\kappa_{j},\\[2pt]
{\rm i}\big(|\boldsymbol{\alpha}^{(n)}|^2 - \kappa_{j}^2\big)^{1/2},\quad
&|\boldsymbol{\alpha}^{(n)}|>\kappa_{j}.
\end{cases}
\end{equation}
Note that $\beta_{1}^{(0)}=\beta=\kappa_{1}\cos\theta_1$. We
assume that $\kappa_{j}\neq |\boldsymbol\alpha_n|$ for all $n\in\mathbb{Z}^2$ to
exclude all possible resonances. Noting \eqref{beta} and using the bounded
outgoing radiation condition, we obtain
\[
\phi^{(n)}(x_3)=\phi^{(n)}(h) e^{ {\rm i}\beta^{(n)}_{1}(x_3-h)},\quad
\boldsymbol\psi^{(n)}(x_3)=\boldsymbol\psi^{(n)}(h) e^{{\rm
i}\beta^{(n)}_{2}(x_3-h)}.
\]
Hence we deduce Rayleigh's expansions of $\phi$ and
$\boldsymbol{\psi}$ for $x_3>h$:
\begin{equation}\label{pfre}
  \phi(\boldsymbol x)=\sum_{n\in\mathbb{Z}^2}\phi^{(n)}(h) e^{{\rm
i}\bigl(\boldsymbol{\alpha}^{(n)}\cdot
\boldsymbol{r}+\beta^{(n)}_{1}(x_3-h)\bigr)},\quad
\boldsymbol{\psi}(\boldsymbol x)=\sum_{n\in\mathbb{Z}^2}\boldsymbol{\psi}^{
(n)}(h) e^{{\rm i}\bigl(\boldsymbol{\alpha}^{(n)}\cdot
\boldsymbol{r}+\beta^{(n)}_{2}(x_3-h)\bigr)}.
\end{equation}
Combining \eqref{pfre} and the Helmholtz decomposition
\eqref{hdv} yields
\begin{align}
 \boldsymbol{v}(\boldsymbol{x})={\rm i}\sum_{n\in\mathbb{Z}^2}
 &\begin{bmatrix}
  \alpha^{(n)}_{1}\\[5pt]
  \alpha^{(n)}_{2}\\[5pt]
  \beta^{(n)}_{1}
 \end{bmatrix}
\phi^{(n)}(h) e^{{\rm i}\bigl(\boldsymbol{\alpha}^{(n)}\cdot\boldsymbol{r}
+\beta^{(n)}_{1}(x_3-h)\bigr)} \nonumber\\\hspace{1cm}
&+\begin{bmatrix}
\alpha^{(n)}_{2}\psi_{3}^{(n)}(h)-\beta^{(n)}_{2}\psi_{2}^{(n)
} (h)\\[5pt]
\beta^{(n)}_{2}\psi_{1}^{(n)}(h)-\alpha^{(n)}_{1}\psi_{3}^{(n)} (h)\\[5pt]
  \alpha^{(n)}_{1}\psi_{2}^{(n)}(h)-\alpha^{(n)}_{2}\psi_{1}^{(n)}(h)
 \end{bmatrix}
 e^{{\rm i}\bigl(\boldsymbol{\alpha}^{(n)}\cdot\boldsymbol{r}
+\beta^{(n)}_{2}(x_3-h)\bigr)}.\label{vre}
\end{align}

On the other hand, as a quasi-biperiodic function, the diffracted field
$\boldsymbol{v}$ has the Fourier series expansion:
\begin{equation}\label{vfe}
 \boldsymbol{v}(\boldsymbol
r,h)=\sum_{n\in\mathbb{Z}^2}\boldsymbol{v}^{(n)}(h) e^{{\rm
i}\boldsymbol{\alpha}^{(n)}\cdot\boldsymbol{r}}.
\end{equation}
It follows from \eqref{vre}--\eqref{vfe} and $\nabla\cdot\boldsymbol{\psi}=0$
that we obtain a linear system of algebraic equations for $\phi^{(n)}(h)$ and
$\psi^{(n)}_k(h)$:
\begin{equation}\label{ls}
 {\rm i}
 \begin{bmatrix}
 \alpha^{(n)}_{1} &                        0 & -\beta^{(n)}_{2} &  \alpha^{(n)}_{2}\\[5pt]
 \alpha^{(n)}_{2} & \beta^{(n)}_{2} &         0 & -\alpha^{(n)}_{1}\\[5pt]
 \beta^{(n)}_{1} &    -\alpha^{(n)}_{2} &  \alpha^{(n)}_{1} &             0\\[5pt]
 0 &          \alpha^{(n)}_{1} &  \alpha^{(n)}_{2} & \beta^{(n)}_{2}
 \end{bmatrix}
 \begin{bmatrix}
  \phi^{(n)}(h)\\[5pt]
  \psi_{1}^{(n)}(h)\\[5pt]
  \psi_{2}^{(n)}(h)\\[5pt]
  \psi_{3}^{(n)}(h)
 \end{bmatrix}=
 \begin{bmatrix}
  v_{1}^{(n)}(h)\\[5pt]
  v_{2}^{(n)}(h)\\[5pt]
  v_{3}^{(n)}(h)\\[5pt]
  0
 \end{bmatrix}.
 \end{equation}
 Solving the above linear system directly via Cramer's rule gives
  \begin{align*}
\phi^{(n)}(h)&=-\frac{\rm i}{\chi^{(n)}}\big(\alpha^{(n)}_{1}  v_{1}^{(n)}(h)
+ \alpha^{(n)}_{2}  v_{2}^{(n)}(h)+ \beta^{(n)}_{2}  v_{3}^{(n)}(h)\big)\\
\psi_{1}^{(n)}(h) &=-\frac{\rm i}{\chi^{(n)}}\big(
\alpha^{(n)}_{1}\alpha^{(n)}_{2}(\beta^{(n)}_{1}-\beta^{(n)}_{2}) 
v_{1}^{(n)}(h)/\kappa^2_{2} \notag\\
 &+\big[(\alpha^{(n)}_{1})^2\beta^{(n)}_{2}+(\alpha^{(n)}_{2})^2\beta^{(n)}_{1}
+ \beta^{(n)}_{1}(\beta^{(n)}_{2})^2\big]  v_{2}^{(n)}(h)/\kappa^2_{2}
-\alpha^{(n)}_{2}  v_{3}^{(n)}(h)\big) \\
\psi_{2}^{(n)}(h) &=-\frac{\rm i}{\chi^{(n)}}\big(
-\big[(\alpha^{(n)}_{1})^2\beta^{(n)}_{1}+(\alpha^{(n)}_{2})^2\beta^{(n)}_{2} +
\beta^{(n)}_{1}(\beta^{(n)}_{2})^2\big]
  v_{1}^{(n)}(h) /\kappa^2_{2} \notag\\
  &-\alpha^{(n)}_{1}\alpha^{(n)}_{2}(\beta^{(n)}_{1}-\beta^{(n)}_{2})v_{2}^{(n)}
(h)/\kappa^2_{2} +\alpha^{(n)}_{1}  v_{3}^{(n)}(h)\big) \\
\psi_{3}^{(n)}(h) &=-\frac{\rm i}{\kappa^2_{2} }\big( \alpha^{(n)}_{2}
v_{1}^{(n)}(h) -\alpha^{(n)}_{1} v_{2}^{(n)}(h)\big),
 \end{align*}
 where
\begin{equation}\label{chi}
\chi^{(n)}=|\boldsymbol{\alpha}^{(n)}|^2+\beta^{(n)}_{1}\beta^{(n)}_{2}.
\end{equation}

Given a vector field $\boldsymbol{v}=(v_1, v_2,v_3)^\top$, we define a
differential operator $\mathscr{D}$ on $\Gamma_h$:
\begin{equation}\label{do}
\mathscr{D}\boldsymbol{v}=\mu\partial_{x_3}\boldsymbol{v}+(\lambda
+\mu)(\nabla\cdot\boldsymbol{v})\boldsymbol
e_3,
\end{equation}
where $\boldsymbol e_3=(0, 0, 1)^\top$. Substituting the Helmholtz decomposition
\eqref{hdv} into \eqref{do} and using \eqref{he}, we get
\[
 \mathscr{D}\boldsymbol{v}=\mu\partial_{x_3}
 (\nabla\phi+\nabla\times\boldsymbol\psi)
-(\lambda+\mu)\kappa_{1}^2\phi \boldsymbol e_3.
\]
It follows from \eqref{vre} that
\begin{align}\label{bv}
  (\mathscr{D}\boldsymbol{v})^{(n)}=-\mu
\begin{bmatrix}
\alpha^{(n)}_{1}\beta^{(n)}_{1} &           0 &    -(\beta^{(n)}_{2})^2 &  \alpha^{(n)}_{2}\beta^{(n)}_{2}\\[5pt]
\alpha^{(n)}_{2}\beta^{(n)}_{1} & (\beta^{(n)}_{2})^2 &         0 & -\alpha^{(n)}_{1}\beta^{(n)}_{2}\\[5pt]
(\beta^{(n)}_{2})^2 &    -\alpha^{(n)}_{2}\beta^{(n)}_{2} &  \alpha^{(n)}_{1}\beta^{(n)}_{2} &         0
\end{bmatrix}\begin{bmatrix}
  \phi^{(n)}(h)\\[5pt]
  \psi_{1}^{(n)}(h)\\[5pt]
  \psi_{2}^{(n)}(h)\\[5pt]
  \psi_{3}^{(n)}(h)
 \end{bmatrix}.
\end{align}

By \eqref{ls} and \eqref{bv}, we deduce the transparent
boundary conditions for the diffracted field:
\[
\mathscr{D}\boldsymbol{v}=\mathscr{T}\boldsymbol{v}
:=\sum_{n\in\mathbb{Z}^2}M^{(n)}
 \boldsymbol{v}^{(n)}(h)e^{{\rm i}\boldsymbol{\alpha}^{(n)}\cdot\boldsymbol{
r}}\quad\text{on} ~ \Gamma,
\]
where the matrix
\begin{align*}
 M^{(n)}&=\frac{{\rm i}\mu}{\chi^{(n)}}\\
& \begin{bmatrix}
 (\alpha^{(n)}_{1})^2(\beta^{(n)}_{1}-\beta^{(n)}_{2})+\beta^{(n)}_{2}\chi^{(n)}
&
\alpha^{(n)}_{1}\alpha^{(n)}_{2}(\beta^{(n)}_{1}-\beta^{(n)}_{2})&\alpha^{(n)}_
{ 1}\beta^{(n)}_{2}(\beta^{(n)}_{1}-\beta^{(n)}_{2})\\[ 5pt]
   \alpha^{(n)}_{1}\alpha^{(n)}_{2}(\beta^{(n)}_{1}-\beta^{(n)}_{2})& (\alpha^{(n)}_{2})^2(\beta^{(n)}_{1}-\beta^{(n)}_{2})+\beta^{(n)}_{2}\chi^{(n)} &
\alpha^{(n)}_{2}\beta^{(n)}_{2}(\beta^{(n)}_{1}-\beta^{(n)}_{2})\\[ 5pt]
-\alpha^{(n)}_{1}\beta^{(n)}_{2}(\beta^{(n)}_{1}-\beta^{(n)}_{2})&
-\alpha^{(n)}_{2}\beta^{(n)}_{2}(\beta^{(n)}_{1}-\beta^{(n)}_{2})& \kappa_{2}^2\beta^{(n)}_{2}
 \end{bmatrix}.
\end{align*}

Equivalently, we have the transparent boundary condition for the total
field $\boldsymbol{u}$ on $\Gamma_h$:
\[
 \mathscr{D}\boldsymbol{u}=\mathscr{T}\boldsymbol{u}+\boldsymbol{g},
\]
where
\[
\boldsymbol{g}=\mathscr{D}\boldsymbol{u}_{\rm
inc}-\mathscr{T}\boldsymbol{u}^{\rm inc}
=-\frac{2{\rm i}\omega^2\beta_{1}^{(0)}}{\kappa_{1}\chi^{(0)}}
(\alpha_1, \alpha_2, -\beta_{2}^{(0)})^\top
e^{{\rm i}(\alpha_1 x_1+\alpha_2 x_2-\beta_{1}^{(0)}b)}.
\]

The scattering problem can be reduced to the following boundary value problem:
\begin{equation}\label{bvp}
 \begin{cases}
  \mu\Delta\boldsymbol{u}+(\lambda+\mu)\nabla
\nabla\cdot\boldsymbol{u}+\omega^2\boldsymbol{u}=0\quad&\text{in}~ \Omega,\\
\mathscr{D}\boldsymbol{u}=\mathscr{T}\boldsymbol{u}+\boldsymbol{g}
\quad&\text{on} ~ \Gamma_h,\\
\boldsymbol{u}=0 \quad&\text{on} ~\Gamma_f.
 \end{cases}
\end{equation}
The weak formulation of \eqref{bvp} reads as follows: Find $\boldsymbol{u}\in
H^1_{\rm qp}(\Omega)^3$ such that
\begin{equation}\label{wp}
 a(\boldsymbol{u}, \boldsymbol{v})=\langle\boldsymbol{g},
\boldsymbol{v}\rangle_{\Gamma_h},\quad\forall\,\boldsymbol{v}\in H^1_{\rm
qp}(\Omega)^3,
\end{equation}
where the sesquilinear form $a: H^1_{\rm qp}(\Omega)^3\times H^1_{\rm
qp}(\Omega)^3\to\mathbb{C}$ is defined by
\begin{align}\label{sf}
 a(\boldsymbol{u}, \boldsymbol{v})=\mu\int_\Omega
\nabla\boldsymbol{u}:\nabla\bar{\boldsymbol
v}\,{\rm d}\boldsymbol{x}+(\lambda+\mu)\int_\Omega (\nabla\cdot\boldsymbol{u}
)(\nabla\cdot\bar{\boldsymbol v})\,{\rm d}\boldsymbol{x}\notag\\
-\omega^2\int_\Omega \boldsymbol{u}\cdot\bar{\boldsymbol v}\, {\rm
d}\boldsymbol{x}- \langle \mathscr{T}\boldsymbol{u},
\boldsymbol{v}\rangle_{\Gamma_h}.
\end{align}
Here $A:B={\rm tr}(A B^\top)$ is the Frobenius inner product of square matrices
$A$ and $B$.

In this paper, we assume that the variational problem \eqref{wp} admits a unique
solution. It follows from the general theory in \cite{ba-73} that there exists a
constant $\gamma_1>0$ such that the following inf-sup condition holds:
\begin{equation}\label{ifc}
\sup_{0\neq \boldsymbol{v}\in H^1_{\rm qp}(\Omega)^3}\frac{|a(\boldsymbol{u},
\boldsymbol{v})|}{\|\boldsymbol{v}\|_{H^1(\Omega)^3}}\geq \gamma_1
\|\boldsymbol{u}\|_{H^1(\Omega)^3},\quad\forall\,\boldsymbol{u}\in H^1_{\rm
qp}(\Omega)^3.
\end{equation}

\subsection{Energy distribution}

We study the energy distribution for the scattering problem. The result will be
used to verify the accuracy of our numerical method for examples where the
analytical solutions are not available. In general, the energy is
distributed away from the scattering surface through propagating wave modes.

Consider the Helmholtz decomposition for the total field:
\begin{equation}\label{hdu}
 \boldsymbol{u}=\nabla\phi^{\rm t} +\nabla\times\boldsymbol{\psi}^{\rm t},
   \quad \nabla\cdot\boldsymbol{\psi}^{\rm t}=0.
\end{equation}
Substituting \eqref{hdu} into \eqref{une}, we may verify that the scalar
potential function $\phi^{\rm t}$ and the vector potential function
$\boldsymbol{\psi}^{\rm t}$ satisfy
\[
\Delta\phi^{\rm t}+\kappa_1^2\phi^{\rm t}=0,\quad
\nabla\times(\nabla\times\boldsymbol{\psi}^{\rm
t})-\kappa_2^2\boldsymbol{\psi}^{\rm t}=0\quad\hbox{in} ~\Omega_f.
\]

We also introduce the Helmholtz decomposition for the incident field
\[
 \boldsymbol{u}^{\rm inc}=\nabla\phi^{\rm
inc} +\nabla\times\boldsymbol{\psi}^{\rm inc}, \quad
\nabla\cdot\boldsymbol{\psi}^{\rm inc}=0,
\]
which gives explicitly that
\[
\phi^{\rm inc}=-\frac{1}{\kappa_1^2}\nabla\cdot\boldsymbol{u}^{\rm
inc}=-\frac{{\rm i}}{\kappa_1}e^{{\rm
i}(\boldsymbol\alpha\cdot\boldsymbol r-\beta x_3)},\quad
\boldsymbol{\psi}^{\rm inc}=\frac{1}{\kappa_2^2}\nabla\times\boldsymbol{u}^{\rm
inc}=0.
\]
Hence we have
\[
\phi^{\rm t}=\phi^{\rm inc}+\phi,\quad
\boldsymbol{\psi}^{\rm t}=\boldsymbol{\psi}.
\]
Using the Rayleigh expansions \eqref{pfre}, we get
\begin{align}
\label{rect}\phi^{\rm t}(\boldsymbol{x})&=a_0 e^{{\rm
i}(\boldsymbol{\alpha}\cdot\boldsymbol{r}-\beta x_3)}
+\sum_{n\in\mathbb{Z}^2}a_1^{(n)} e^{{\rm
i}\bigl(\boldsymbol{\alpha}^{(n)}\cdot
\boldsymbol{r}+\beta^{(n)}_{1}x_3\bigr)}\\
\label{rest}\boldsymbol{\psi}^{\rm t}(\boldsymbol{x})&=\sum_{n\in\mathbb{Z}^2}
\boldsymbol b^{(n)} e^{{\rm i}\bigl(\boldsymbol{\alpha}^{(n)}\cdot
\boldsymbol{r}+\beta^{(n)}_{2}x_3\bigr)},
\end{align}
where
\[
a_0=-\frac{{\rm i}}{\kappa_1},\quad a_1^{(n)}=\phi^{(n)}(h) e^{-{\rm
i}\beta^{(n)}_{1}h},\quad \boldsymbol b^{(n)}=\boldsymbol{\psi}^{
(n)}(h) e^{-{\rm i}\beta^{(n)}_{2}h}.
\]

The grating efficiency is defined by
\begin{equation}\label{ge}
e_1^{(n)}=\frac{\beta^{(n)}_{1}|a_1^{(n)}|^2}{\beta|a_0|^2},\quad
e_2^{(n)}=\frac{\beta^{(n)}_{2}|\boldsymbol b^{(n)}|^2}{\beta|a_0|^2},
\end{equation}
where $e_1^{(n)}$ and $e_2^{(n)}$ are the efficiency of the $n$-th order
reflected modes for the compressional wave and the shear wave, respectively. In
practice, the grating efficiency \eqref{ge} can be computed from \eqref{ls}
once the scattering problem is solved and the diffracted field $\boldsymbol v$
is available on $\Gamma_h$.

\begin{theo}\label{coe}
The total energy is conserved, i.e.,
\[
\sum_{n\in U_1} e_1^{(n)}+\sum_{n\in U_2} e_2^{(n)}=1,
\]
where $U_j=\{n: |\boldsymbol\alpha^{(n)}|\leq\kappa_j\}$.
\end{theo}

\begin{proof}
It follows from the boundary condition \eqref{bc} and the Helmholtz
decomposition \eqref{hdu} that
\[
 \nabla\phi^{\rm t}+\nabla\times\boldsymbol\psi^{\rm t}=0\quad\text{on} ~
\Gamma_f,
\]
which gives
\[
 \boldsymbol\nu\cdot\nabla\phi^{\rm
t}+\boldsymbol\nu\cdot(\nabla\times\boldsymbol\psi^{\rm t})=0,\quad
\boldsymbol\nu\times\nabla\phi^{\rm
t}+\boldsymbol\nu\times(\nabla\times\boldsymbol\psi^{\rm t})=0.
\]
Here $\boldsymbol\nu$ is the unit normal vector on $\Gamma_f$.

Consider the following coupled problem:
\begin{equation}\label{cp}
 \begin{cases}
  \Delta\phi^{\rm t}+\kappa_1^2\phi^{\rm t}=0,\quad
\nabla\times(\nabla\times\boldsymbol\psi^{\rm
t})-\kappa_2^2\boldsymbol\psi^{\rm t}=0\quad&\text{in} ~ \Omega,\\
\boldsymbol\nu\cdot\nabla\phi^{\rm
t}+\boldsymbol\nu\cdot(\nabla\times\boldsymbol\psi^{\rm t})=0,\quad
\boldsymbol\nu\times\nabla\phi^{\rm
t}+\boldsymbol\nu\times(\nabla\times\boldsymbol\psi^{\rm t})=0 \quad&\text{on} ~
\Gamma_f.
 \end{cases}
\end{equation}
It is clear to note that $(\phi^{\rm t}, \boldsymbol\psi^{\rm t})$ also
satisfies the problem \eqref{cp} since the wavenumbers $\kappa_j$ are real.
Using Green's theorem and quasi-periodicity of the solution, we get
\begin{align}
 0=&\int_\Omega (\bar\phi^{\rm t}\Delta\phi^{\rm t}-\phi^{\rm
t}\Delta\bar\phi^{\rm t}) {\rm d}\boldsymbol x-\int_\Omega
(\bar{\boldsymbol\psi}^{\rm t}\cdot\nabla\times(\nabla\times\boldsymbol\psi^{\rm
t})-\boldsymbol\psi^{\rm t}\cdot\nabla\times(\nabla\times\bar{\boldsymbol\psi}^{
\rm t})){\rm d}\boldsymbol x\notag\\
=&\int_{\Gamma_f} (\bar\phi^{\rm t}\partial_{\boldsymbol\nu}\phi^{\rm
t}-\phi^{\rm t}\partial_{\boldsymbol\nu}\bar\phi^{\rm t}){\rm d}\gamma
-\int_{\Gamma_f}
(\bar{\boldsymbol\psi}^{\rm
t}\cdot (\boldsymbol\nu\times\nabla\times\boldsymbol\psi^{\rm
t})-\boldsymbol\psi^{\rm t}\cdot(\boldsymbol\nu\times\nabla\times\bar{
\boldsymbol\psi}^{\rm t})){\rm d}\gamma\notag\\
\label{gt}&+\int_{\Gamma_h} (\bar\phi^{\rm t}\partial_{x_3}\phi^{\rm
t}-\phi^{\rm t}\partial_{x_3}\bar\phi^{\rm t}){\rm d}\boldsymbol r
-\int_{\Gamma_h} (\bar{\boldsymbol\psi}^{\rm t}\cdot (\boldsymbol e_3
\times\nabla\times\boldsymbol\psi^{\rm t})-\boldsymbol\psi^{\rm
t}\cdot(\boldsymbol e_3\times\nabla\times\bar{\boldsymbol\psi}^{\rm t})){\rm
d}\boldsymbol r.
\end{align}
It follows from the integration by parts and the boundary conditions in
\eqref{cp} that
\begin{align*}
 \int_{\Gamma_f}\partial_{\boldsymbol\nu}\phi^{\rm t}\bar{\phi}^{\rm t}{\rm
d}\gamma=-\int_{\Gamma_f}\boldsymbol\nu\cdot(\nabla\times\boldsymbol\psi^{\rm
t})\bar{\phi}^{\rm t}{\rm d}\gamma=\int_{\Gamma_f}\boldsymbol\psi^{\rm t}\cdot(
\boldsymbol\nu\times\nabla\bar{\phi}^{\rm t}){\rm
d}\gamma=-\int_{\Gamma_f}\boldsymbol\psi^{\rm t}
\cdot(\boldsymbol\nu\times (\nabla\times\boldsymbol{\bar\psi}^{\rm t})){\rm
d}\gamma,
\end{align*}
which gives after taking the imaginary part of \eqref{gt} that
\begin{equation}\label{ip}
 {\rm Im}\int_{\Gamma_h} (\bar\phi^{\rm t}\partial_{x_3}\phi^{\rm
t}-\bar{\boldsymbol\psi}^{\rm t}\cdot (\boldsymbol
e_3\times\nabla\times\boldsymbol\psi^{\rm t})){\rm d}\boldsymbol r=0.
\end{equation}

Let $\Delta_j^{(n)}=|\kappa_j^2-|\boldsymbol\alpha^{(n)}|^2|^{1/2}$.  It is
clear to note that $\beta_j^{(n)}=\Delta_j^{(n)}$ for $n\in U_j$ and
$\beta_j^{(n)}={\rm i}\Delta_j^{(n)}$ for $n\notin U_j$. It follows from
\eqref{rect} and \eqref{rest} that we have
\begin{align*}
\phi^{\rm t}(\boldsymbol{r}, h)&=a_0 e^{{\rm
i}(\boldsymbol{\alpha}\cdot\boldsymbol{r}-\beta h)}
+\sum_{n\in U_1}a_1^{(n)} e^{\big({\rm
i}\boldsymbol{\alpha}^{(n)}\cdot \boldsymbol{r}+{\rm
i}\Delta^{(n)}_{1}h\big)}+\sum_{n\notin U_1}a_1^{ (n) } e^{\big({\rm
i}\boldsymbol{\alpha}^{(n)}\cdot
\boldsymbol{r}-\Delta^{(n)}_{1} h\big)},\\
\boldsymbol{\psi}^{\rm t}(\boldsymbol{r}, h)&=\sum_{n\in U_2}
\boldsymbol b^{(n)} e^{\bigl({\rm i}\boldsymbol{\alpha}^{(n)}\cdot
\boldsymbol{r}+{\rm i}\Delta^{(n)}_{2} h\bigr)}+\sum_{n\notin U_2}
\boldsymbol b^{(n)} e^{\bigl({\rm i}\boldsymbol{\alpha}^{(n)}\cdot
\boldsymbol{r}-\Delta^{(n)}_{2} h\bigr)},
\end{align*}
and
\begin{align*}
 \partial_{x_3}\phi^{\rm t}(\boldsymbol{r}, h)&=-{\rm i}\beta a_0 e^{{\rm
i}(\boldsymbol{\alpha}\cdot\boldsymbol{r}-\beta h)}
+\sum_{n\in U_1}{\rm i}\Delta_1^{(n)}a_1^{(n)} e^{\big({\rm
i}\boldsymbol{\alpha}^{(n)}\cdot \boldsymbol{r}+{\rm
i}\Delta^{(n)}_{1}h\big)}-\sum_{n\notin U_1}\Delta_1^{(n)}a_1^{ (n) }
e^{\big({\rm i}\boldsymbol{\alpha}^{(n)}\cdot
\boldsymbol{r}-\Delta^{(n)}_{1} h\big)},\\
\boldsymbol e_3\times\nabla\times\boldsymbol\psi^{\rm t}(\boldsymbol r,
h)&=\sum_{n\in U_2}\begin{bmatrix}
                    {\rm i}\alpha_1^{(n)}b_3^{(n)}-{\rm
i}\Delta_2^{(n)}b^{(n)}_1\\
{\rm i}\alpha_2^{(n)}b_3^{(n)}-{\rm
i}\Delta_2^{(n)}b^{(n)}_2\\
0
                   \end{bmatrix} e^{\bigl({\rm i}\boldsymbol{\alpha}^{(n)}\cdot
\boldsymbol{r}+{\rm i}\Delta^{(n)}_{2} h\bigr)}\\
&\hspace{2cm}+\sum_{n\notin U_2}\begin{bmatrix}
                    {\rm i}\alpha_1^{(n)}b_3^{(n)}+\Delta_2^{(n)}b^{(n)}_1\\
{\rm i}\alpha_2^{(n)}b_3^{(n)}+\Delta_2^{(n)}b^{(n)}_2\\
0
\end{bmatrix} e^{\bigl({\rm i}\boldsymbol{\alpha}^{(n)}\cdot
\boldsymbol{r}-\Delta^{(n)}_{2} h\bigr)},
\end{align*}
where $\boldsymbol b^{(n)}=(b^{(n)}_1, b^{(n)}_2, b^{(n)}_3)^\top$. 
Substituting the above four functions into \eqref{ip}, using the
orthogonality of the Fourier series and the divergence free condition, we
obtain
\[
 \beta|a_0|^2=\sum_{n\in U_1}\Delta_1^{(n)}|a_1^{(n)}|^2+\sum_{n\in
U_2}\Delta_2^{(n)}|\boldsymbol b^{(n)}|^2,
\]
which completes the proof.
\end{proof}

\section{The PML problem}

In this section, we introduce the PML formulation for the scattering problem and
establish the well-posedness of the PML problem. An error estimate will be shown
for the solutions between the original scattering problem and the
PML problem.

\subsection{PML formulation}

Now we turn to the introduction of an absorbing PML layer. The domain $\Omega$
is covered by a PML layer of thickness $\delta$ in $\Omega_h$. Let
$\rho(\tau)=\rho_1(\tau)+{\rm i}\rho_2(\tau)$ be the PML function which is
continuous and satisfies
\[
 \rho_1=1, \quad \rho_2=0 \quad\text{for}~ \tau<h\quad\text{and}\quad
\rho_1\geq 1,\quad \rho_2>0\quad\text{otherwise}.
\]
We introduce the PML by complex coordinate stretching:
\begin{equation}\label{cs}
 \hat{x}_3=\int_0^{x_3} \rho(\tau) {\rm d}\tau.
\end{equation}

Let $\hat{\boldsymbol x}=(\boldsymbol r,\hat{x}_3)$. Introduce the new field
\begin{equation}\label{nf}
 \hat{\boldsymbol u}(\boldsymbol{x})=\begin{cases}
\boldsymbol{u}^{\rm inc}(\boldsymbol{x})+(\boldsymbol{u}(\hat{\boldsymbol
x})-\boldsymbol{u}^{\rm inc}(\hat{\boldsymbol x})),\quad&
\boldsymbol{x}\in\Omega_h,\\
\boldsymbol{u}(\hat{\boldsymbol x}) ,\quad
&\boldsymbol{x}\in\Omega.
 \end{cases}
\end{equation}
It is clear to note that $\hat{\boldsymbol u}({\boldsymbol
x})=\boldsymbol{u}(\boldsymbol{x})$ in $\Omega$ since $\hat{\boldsymbol
x}=\boldsymbol{x}$ in $\Omega$. It can be verified from \eqref{une} and
\eqref{cs} that $\hat{\boldsymbol u}$ satisfies
\[
 \mathscr{L}(\hat{\boldsymbol u}-\boldsymbol{u}^{\rm inc})=0\quad\text{in}
~\Omega_f.
\]
Here the PML differential operator
\[
  \mathscr{L}\boldsymbol{u}=(w_1, w_2, w_3)^\top,
\]
where
\begin{align*}
 w_1=& (\lambda+2\mu)\partial_{x_1x_1}^2 u_1+\mu(\partial_{x_2x_2}^2
u_1+\rho^{-1}(x_3)\partial_{x_3}
(\rho^{-1}(x_3)\partial_{x_3} u_1))\\
&+(\lambda+\mu)(\partial^2_{x_1x_2} u_2+\rho^{-1}(x_3)\partial^2_{x_1x_3}
u_3)+\omega^2u_1,\\
w_2= &(\lambda+2\mu)\partial_{x_2x_2}^2 u_2+\mu(\partial_{x_1x_1}^2
u_2+\rho^{-1}(x_3)\partial_{x_3}(\rho^{-1}(x_3)\partial_{x_3}
u_2))\\
&+(\lambda+\mu)(\partial^2_{x_1x_2} u_1 +\rho^{-1}(x_3)\partial^2_{x_2x_3}
u_3)+\omega^2u_2\\
w_3=&(\lambda+2\mu)\rho^{-1}(x_3)\partial_{x_3}(\rho^{-1}(x_3)\partial_{x_3}
u_3)+\mu(\partial_{x_1x_1}^2 u_3+\partial_{x_2x_2}^2
u_3))\\
&+(\lambda+\mu)\rho^{-1}(x_3)(\partial^2_{x_1x_3} u_1 +\partial^2_{x_2x_3}
u_2)+\omega^2u_3.
\end{align*}

Define the PML regions
\begin{align*}
 \Omega^{\rm PML}&=\{\boldsymbol{x}\in\mathbb{R}^3: 0<x_1<\Lambda_1,\,
0<x_2<\Lambda_2,\, h<x_3<h+\delta\}.
\end{align*}
It is clear to note from \eqref{nf} that the outgoing wave
$\hat{\boldsymbol u}(\boldsymbol{x})-\boldsymbol{u}^{\rm inc}(\boldsymbol{x})$
in $\Omega_h$ decay exponentially as $x_3\to +\infty$. Therefore, the
homogeneous Dirichlet boundary condition can be
imposed on
\[
 \Gamma^{\rm PML}=\{\boldsymbol{x}\in\mathbb{R}^3:
0<x_1<\Lambda_1,\,0<x_2<\Lambda_2,\,x_3=h+\delta\}
\]
to truncate the PML problem. Define the computational domain for the PML problem
$D=\Omega\cup\Omega^{\rm PML}$. We arrive at
the following truncated PML
problem: Find a quasi-periodic solution $\hat{\boldsymbol u}$ such that
\begin{equation}\label{pmlp}
 \begin{cases}
 \mathscr{L}\hat{\boldsymbol u}=\boldsymbol{g}&\quad\text{in} ~ D,\\
  \hat{\boldsymbol u}=\boldsymbol{u}^{\rm inc}&\quad\text{on} ~
\Gamma^{\rm PML},\\
\hat{\boldsymbol u}=0&\quad\text{on} ~ \Gamma_f,
 \end{cases}
\end{equation}
where
\[
 \boldsymbol{g}=\begin{cases}
                 \mathscr{L}\boldsymbol{u}^{\rm inc}&\quad\text{in} ~
\Omega^{\rm PML},\\
                 0&\quad\text{in} ~ \Omega.
                \end{cases}
\]

Define $H^1_{0,\rm qp}(D)=\{u\in H^1_{\rm qp}(D):
u=0~\text{on}~\Gamma^{\rm PML}\cup\Gamma_f\}$.
The weak formulation of the PML problem \eqref{pmlp} reads as
follows: Find $\hat{\boldsymbol u}\in H^1_{\rm qp}(D)^3$ such that
$\hat{\boldsymbol u}=\boldsymbol{u}^{\rm inc}$ on $\Gamma^{\rm PML}$ and
\begin{equation}\label{twp}
 b_D(\hat{\boldsymbol u}, \boldsymbol{v})=-\int_D
\boldsymbol{g}\cdot\bar{\boldsymbol v}{\rm d}\boldsymbol{x},\quad\forall\,
\boldsymbol{v}\in H^1_{0, \rm qp}(D)^3.
\end{equation}
Here for any domain $G\subset\mathbb{R}^3$, the sesquilinear form $b_G: H^1_{\rm qp}(G)^3\times
H^1_{\rm qp}(G)^3\to\mathbb{C}$ is defined by
\begin{align*}
 b_G(\boldsymbol{u}, \boldsymbol{v})=\int_G & (\lambda+2\mu)
(\partial_{x_1} u_1\partial_{x_1}\bar{v}_1+\partial_{x_2} u_2\partial_{x_2}\bar{v}_2+(\rho^{-1})^2\partial_{x_3}
u_3\partial_{x_3}\bar{v}_3)\\
&+\mu(\partial_{x_2}u_1\partial_{x_2}\bar{v}_1
          +\partial_{x_1}u_2\partial_{x_1}\bar{v}_2
          +\partial_{x_1}u_3\partial_{x_1}\bar{v}_3+\partial_{x_2} u_3\partial_{x_2}\bar{v}_3)\\
&+\mu(\rho^{-1})^2(\partial_{x_3} u_1\partial_{x_3}\bar{v}_1
+\partial_{x_3} u_2\partial_{x_3}\bar{v}_2)+(\lambda+\mu)(\partial_{x_2}
u_2\partial_{x_1}\bar{v}_1
               +\partial_{x_1}u_1\partial_{x_2}\bar{v}_2)\\
&+(\lambda+\mu)\rho^{-1}(\partial_{x_3} u_3\partial_{x_1}\bar{v}_1
               +\partial_{x_3}u_3\partial_{x_2}\bar{v}_2
               +\partial_{x_1}u_1\partial_{x_3}\bar{v}_3
               +\partial_{x_2}u_2\partial_{x_3}\bar{v}_3)\\
&-\omega^2(u_1\bar{v}_1+u_2\bar{v}_2+u_3\bar{v}_3)\,{\rm
d}\boldsymbol{x}.
\end{align*}

We will reformulate the variational problem \eqref{twp} in the domain
$D$ into an equivalent variational formulation in the domain $\Omega$, and
discuss the existence and uniqueness of the weak solution to the equivalent weak
formulation. To do so, we need to introduce the transparent boundary condition
for the truncated PML problem.

\subsection{Transparent boundary condition of the PML problem}

Let $\hat{\boldsymbol v}(\boldsymbol{x})=\boldsymbol{v}(\hat{\boldsymbol
x})=\boldsymbol{u}(\hat{\boldsymbol x})-\boldsymbol{u}_{\rm
inc}(\hat{\boldsymbol x})$ in $\Omega^{\rm PML}$. It is clear to note that $\hat{\boldsymbol v}$
satisfies the Navier equation in the complex coordinate
\begin{equation}\label{cvne}
\mu\Delta_{\hat{\boldsymbol x}}\hat{\boldsymbol v}+(\lambda
+\mu)\nabla_{\hat{\boldsymbol x}}\nabla_{\hat{\boldsymbol
x}}\cdot\hat{\boldsymbol v} +\omega^2\hat{\boldsymbol v}=0\quad\text{in} ~
\Omega^{\rm PML},
\end{equation}
where $\nabla_{\hat{\boldsymbol x}}=(\partial_{x_1},\partial_{x_2},
\partial_{{\hat x}_3})^\top$ with $\partial_{{\hat
x}_3}=\rho^{-1}(x_3)\partial_{x_3}$.

We introduce the Helmholtz decomposition for the solution of \eqref{cvne}:
\begin{equation}\label{chdv}
 \hat{\boldsymbol v}=\nabla_{\hat{\boldsymbol x}}\hat{\phi}
 +\nabla_{\hat{\boldsymbol x}}\times\hat{\boldsymbol{\psi}},\quad
 \nabla_{\hat{\boldsymbol x}}\cdot\hat{\boldsymbol{\psi}}=0,
\end{equation}
Plugging \eqref{chdv} into \eqref{cvne} gives
\begin{equation}\label{che}
 \Delta_{\hat{\boldsymbol{x}}}\hat{\phi}
+\kappa_{1}^2\hat{\phi}=0,\quad
\Delta_{\hat{\boldsymbol{x}}}\hat{\boldsymbol{\psi}}
+\kappa_{2}^2\hat{\boldsymbol{\psi}}=0.
 \end{equation}
Due to the quasi-periodicity of the solution, we have the Fourier series
expansions
\[
 \hat{\phi}(\boldsymbol x)=\sum_{n\in\mathbb{Z}^2}\hat{\phi}^{(n)}(x_3)
 e^{{\rm i}\boldsymbol{\alpha}^{(n)}\cdot\boldsymbol{r}},
 \]
 and
 \[
 \hat{\boldsymbol{\psi}}(\boldsymbol x)=\sum_{n\in\mathbb{Z}^2}
 \big(\hat{\psi}^{(n)}_{1}(x_3),\hat{\psi}^{(n)}_{2}(x_3),\hat{\psi}^{(n)}_{3} (x_3)\big)^\top
 e^{{\rm i}\boldsymbol{\alpha}^{(n)}\cdot\boldsymbol{r}}.
 \]
Substituting the above Fourier series expansions into \eqref{che} yields
\[
 \rho^{-1}\frac{\rm d}{{\rm d}x_3}\Bigl(\rho^{-1}\frac{{\rm
d}}{{\rm d}x_3}\hat{\phi}^{(n)}(x_3)\Bigr)+(\beta_{1}^{(n)})^2
\hat{\phi}^{(n)}(x_3)=0
\]
and
\[
 \rho^{-1}\frac{\rm d}{{\rm d}x_3}\Bigl(\rho^{-1}\frac{{\rm
d}}{{\rm d}x_3}\hat{\psi}_{k}^{(n)}(x_3)\Bigr)+(\beta_{2}^{(n)})^2
\hat{\psi}_{k}^{(n)}(x_3)=0,\quad k=1,2,3.
\]
The general solutions of the above equations are
\begin{equation}\label{slu}
\begin{cases}
 \hat{\phi}^{(n)}(x_3)=A^{(n)} e^{{\rm
i}\beta_{1}^{(n)}\int_{h}^{x_3}\rho(\tau){\rm d}\tau} + B^{(n)}e^{-{\rm
i}\beta_{1}^{(n)}\int_{h}^{x_3}\rho(\tau){\rm d}\tau},\\
 \hat{\psi}_{k}^{(n)}(x_3)=C_{k}^{(n)} e^{{\rm
i}\beta_{2}^{(n)}\int_{h}^{x_3}\rho(\tau){\rm d}\tau} +
D_{k}^{(n)}e^{-{\rm i}\beta_{2}^{(n)}\int_{h}^{x_3} \rho(\tau){\rm
d}\tau}.
\end{cases}
\end{equation}

Define
\begin{equation}\label{zeta}
\zeta=\int_{h}^{h+\delta}\rho(\tau){\rm d}\tau,\quad
\zeta(x_3)=\int_{h}^{x_3}\rho(\tau){\rm d}\tau.
\end{equation}
The coefficients $A^{(n)}$, $B^{(n)}$, $C_{k}^{(n)}, D_{k}^{(n)}$
can be uniquely determined by solving the following linear system:
\begin{equation}\label{le}
  \mathbf{A}^{(n)}\mathbf{X}^{(n)}=\mathbf{V}^{(n)},
\end{equation}
where
\begin{align*}
&\mathbf{X}^{(n)}=\big(A^{(n)}, B^{(n)}, C_{1}^{(n)}, D_{1}^{(n)},
C_{2}^{(n)}, D_{2}^{(n)}, C_{3}^{(n)}, D_{3}^{(n)}\big)^\top,\\
&\mathbf{V}^{(n)}=-{\rm i}\big(v_1^{(n)}(h),
v_2^{(n)}(h), v^{(n)}_{3}(h), 0, 0, 0, 0, 0\big)^\top,
\end{align*}
and
\begin{equation*}
  \mathbf{A}^{(n)}=
  \begin{bmatrix}
    A_{11}^{(n)}&A_{12}^{(n)}\\[5pt]
    A_{21}^{(n)}&A_{22}^{(n)}
  \end{bmatrix}.
\end{equation*}
Here the block matrices are
\begin{equation*}
  A_{1 1}^{(n)}=
  \begin{bmatrix}
    \alpha^{(n)}_1 & \alpha^{(n)}_1&0&0\\[5pt]
    \alpha^{(n)}_2 & \alpha^{(n)}_2&\beta_{2}^{(n)}&-\beta_{2}^{(n)}\\[5pt]
    \beta_{1}^{(n)} & -\beta_{1}^{(n)}&-\alpha^{(n)}_2&-\alpha^{(n)}_2\\[5pt]
    \alpha^{(n)}_1 e^{{\rm i}\beta_{1}^{(n)}\zeta} &
 \alpha^{(n)}_1 e^{-{\rm i}\beta_{1}^{(n)}\zeta}& 0&0
  \end{bmatrix},
\end{equation*}
\begin{equation*}
  A_{1 2}^{(n)}=
  \begin{bmatrix}
    -\beta_{2}^{(n)}& \beta_{2}^{(n)}& \alpha^{(n)}_{2} &
\alpha^{(n)}_{2}\\[5pt]
    0& 0& -\alpha_{1}^{(n)} & -\alpha_{1}^{(n)}\\[5pt]
    \alpha_{1}^{(n)}& \alpha_{1}^{(n)}& 0 & 0\\[5pt]
    -\beta_{2}^{(n)}e^{{\rm i}\beta_{2}^{(n)}\zeta} &
    \beta_{2}^{(n)}e^{-{\rm i}\beta_{2}^{(n)}\zeta}&
    \alpha_{2}^{(n)}e^{{\rm i}\beta_{2}^{(n)}\zeta}&
    \alpha_{2}^{(n)}e^{-{\rm i}\beta_{2}^{(n)}\zeta}
  \end{bmatrix},
\end{equation*}
\begin{equation*}
  A_{2 1}^{(n)}=
  \begin{bmatrix}
    \alpha_{2}^{(n)}e^{{\rm i}\beta_{1}^{(n)}\zeta}&
    \alpha_{2}^{(n)}e^{-{\rm i}\beta_{1}^{(n)}\zeta}&
    \beta_{2}^{(n)}e^{{\rm i}\beta_{2}^{(n)}\zeta} &
    -\beta_{2}^{(n)}e^{-{\rm i}\beta_{2}^{(n)}\zeta}\\[5pt]
    \beta_{1}^{(n)}e^{{\rm i}\beta_{1}^{(n)}\zeta}&
    -\beta_{1}^{(n)}e^{-{\rm i}\beta_{1}^{(n)}\zeta}&
    -\alpha_{2}e^{{\rm i}\beta_{2}^{(n)}\zeta} &
    -\alpha_{2}e^{-{\rm i}\beta_{2}^{(n)}\zeta}\\[5pt]
    0& 0& \alpha_{1}^{(n)} & \alpha_{1}^{(n)}\\[5pt]
    0& 0& \alpha_{1}^{(n)}e^{{\rm i}\beta_{2}^{(n)}\zeta} &
    \alpha_{1}^{(n)}e^{-{\rm i}\beta_{2}^{(n)}\zeta}\\[5pt]
  \end{bmatrix},
\end{equation*}
\begin{equation*}
  A_{2 2}^{(n)}=
  \begin{bmatrix}
    0& 0& -\alpha_{1}^{(n)}e^{{\rm i}\beta_{2}^{(n)}\zeta} &
    -\alpha_{1}^{(n)}e^{-{\rm i}\beta_{2}^{(n)}\zeta}\\[5pt]
    \alpha_{1}^{(n)}e^{{\rm i}\beta_{2}^{(n)}\zeta} &
    \alpha_{1}^{(n)}e^{-{\rm i}\beta_{2}^{(n)}\zeta}&     0&
0\\[5pt]
    \alpha_{2}^{(n)}& \alpha_{2}^{(n)}& \beta_{2}^{(n)} &
-\beta_{2}^{(n)}\\[5pt]
    \alpha_{2}^{(n)}e^{{\rm i}\beta_{2}^{(n)}\zeta} &
    \alpha_2^{(n)}e^{-{\rm i}\beta_{2}^{(n)}\zeta}&
    \beta_{2}^{(n)}e^{{\rm i}\beta_{2}^{(n)}\zeta} &
    -\beta_{2}^{(n)}e^{-{\rm i}\beta_{2}^{(n)}\zeta}
  \end{bmatrix}.
\end{equation*}
To obtain the above linear system \eqref{le}, we have used the Helmholtz
decomposition \eqref{chdv} and the homogeneous Dirichlet boundary condition
\[
\hat{\boldsymbol v}(\boldsymbol r, h+\delta)=0\quad\text{on} ~
\Gamma^{\rm PML}
\]
due to the PML absorbing layer.

Using the Helmholtz decomposition \eqref{chdv} and \eqref{slu}, we get
\begin{align}\label{cvre}
&\hat{\boldsymbol v}(\boldsymbol x)={\rm i} \sum_{n\in\mathbb{Z}^2}
\begin{bmatrix}
          \alpha_{1}^{(n)}\\[5pt]
          \alpha_{2}^{(n)}\\[5pt]
          \beta_{1}^{(n)}
         \end{bmatrix}
A^{(n)}e^{{\rm i}\big(\boldsymbol{\alpha}^{(n)}\cdot\boldsymbol{r}
+\beta_{1}^{(n)}\int_{h}^{x_3}\rho(\tau){\rm d}\tau\big)}
+\begin{bmatrix}
          \alpha_{1}^{(n)}\\[5pt]
          \alpha_{2}^{(n)}\\[5pt]
          -\beta_{1}^{(n)}
     \end{bmatrix}
B^{(n)}e^{{\rm i}\big(\boldsymbol{\alpha}^{(n)}\cdot\boldsymbol{r}
-\beta_{1}^{(n)}\int_{h}^{x_3}\rho(\tau){\rm d}\tau\big)}\notag\\
 &+\begin{bmatrix}
\alpha_{2}^{(n)}C_{ 3}^{(n)}-\beta_{2}^{(n)}C_{ 2}^{(n)}\\[5pt]
          \beta_{2}^{(n)}C_{ 1}^{(n)}-\alpha_{1}^{(n)}C_{ 3}^{(n)}\\[5pt]
                      \alpha_{1}^{(n)}C_{
2}^{(n)}-\alpha_{2}^{(n)}C_{ 1}^{(n)}
                     \end{bmatrix}
e^{{\rm i}\big(\boldsymbol{\alpha}^{(n)}\cdot\boldsymbol{r}
+\beta_{2}^{(n)}\int_{h}^{x_3}\rho(\tau){\rm d}\tau\big)}
+\begin{bmatrix}
\alpha_{2}^{(n)}D_{ 3}^{(n)}+\beta_{2}^{(n)}D_{ 2}^{(n)}\\[5pt]
-\beta_{2}^{(n)}D_{ 1}^{(n)}-\alpha_{1}^{(n)}D_{ 3}^{(n)}\\[5pt]
\alpha_{1}^{(n)}D_{ 2}^{(n)}-\alpha_{2}^{(n)}D_{ 1}^{(n)}
\end{bmatrix}
e^{{\rm i}\big(\boldsymbol{\alpha}^{(n)}\cdot\boldsymbol{r}
-\beta_{2}^{(n)}\int_{h}^{x_3}\rho(\tau){\rm d}\tau\big)}.
\end{align}

It follows from \eqref{cvre} that we have
\[
\mathscr{D}\hat{\boldsymbol v}=\mu\partial_{x_3}\hat{
\boldsymbol v}+(\lambda+\mu)(\nabla\cdot\hat{\boldsymbol
v})\boldsymbol e_3=\sum\limits_{n\in\mathbb{Z}^2}\mu\mathbf{P}^{(n)}
\mathbf{X}^{(n)}e^{{\rm
i}\boldsymbol{\alpha}^{(n)}\cdot\boldsymbol{x}}\quad\text{on}~ \Gamma_h,
\]
where
\begin{align*}
&\mathbf{P}^{(n)}=\\
&\begin{bmatrix}
  -\alpha_{1}^{(n)}\beta_{1}^{(n)}& \alpha_{1}^{(n)}\beta_{1}^{(n)}&
                              0 &                             0&
  (\beta_{2}^{(n)})^2         &         (\beta_{2}^{(n)})^2&
  -\alpha_{2}^{(n)}\beta_{2}^{(n)}& \alpha_{2}^{(n)}\beta_{2}^{(n)}\\[5pt]
  -\alpha_{2}^{(n)}\beta_{1}^{(n)}& \alpha_{2}^{(n)}\beta_{1}^{(n)}&
         -(\beta_{2}^{(n)})^2 &        -(\beta_{2}^{(n)})^2&
                               0&                             0&
   \alpha_{1}^{(n)}\beta_{2}^{(n)}&-\alpha_{1}^{(n)}\beta_{2}^{(n)}\\[5pt]
          -(\beta_{2}^{(n)})^2&        -(\beta_{2}^{(n)})^2&
   \alpha_{2}^{(n)}\beta_{2}^{(n)}&-\alpha_{2}^{(n)}\beta_{2}^{(n)}&
  -\alpha_{1}^{(n)}\beta_{2}^{(n)}& \alpha_{1}^{(n)}\beta_{2}^{(n)}&
                               0&                             0\\[5pt]
\end{bmatrix}.
\end{align*}
Combining \eqref{cvre} and \eqref{le}, we derive the transparent boundary
condition for the PML problem:
\[
\mathscr{D}\hat{\boldsymbol v}=\mathscr{T}^{\rm
PML}\hat{\boldsymbol v}:=\sum_{n\in\mathbb{Z}^2}\hat{\mathbf M}^{(n)}
\hat{\boldsymbol v}^{(n)}(b)e^{{\rm
i}\boldsymbol{\alpha}^{(n)}\cdot\boldsymbol{r}}\quad\text{on} ~ \Gamma_h,
\]
where the matrix
\[
\hat{\mathbf M}^{(n)}=
 \begin{bmatrix}
   \hat{m}_{ 11}^{(n)}&\hat{m}_{ 12}^{(n)}&\hat{m}_{ 13}^{(n)}\\[5pt]
   \hat{m}_{ 21}^{(n)}&\hat{m}_{ 22}^{(n)}&\hat{m}_{ 23}^{(n)}\\[5pt]
   \hat{m}_{ 31}^{(n)}&\hat{m}_{ 32}^{(n)}&\hat{m}_{ 33}^{(n)}
 \end{bmatrix}.
\]
Here the entries  of $\hat{\mathbf M}^{(n)}$ are
\begin{align*}
  \hat{m}_{ 11}^{(n)}=&\frac{{\rm i}\mu}{\chi^{(n)}\hat{\chi}^{(n)}}
\Big[\chi^{(n)}\big( (\alpha_{1}^{(n)})^2(\beta_{1}^{(n)}-\beta_{2}^{(n)})+\beta_{2}^{(n)}\chi^{(n)}\big)
  (\varepsilon^{(n)}+1)\\
&\quad+4(\alpha_{2}^{(n)})^2\beta_{1}^{(n)}(\beta_{2}^{(n)})^2\theta^{(n)}
(\varepsilon^{(n)}+1)-2(\alpha_{1}^{(n)})^2\beta_{1}^{(n)}\kappa_{2}^2 \eta^{(n)}\Big],\\
  \hat{m}_{ 12}^{(n)}=&\hat{m}_{2 1}^{(n)}=\frac{{\rm
i}\mu\alpha_{1}^{(n)}\alpha_{2}^{(n)}}{\chi^{(n)}\hat{\chi}^{(n)}}
  \Big[\chi^{(n)}(\beta_{1}^{(n)}-\beta_{2}^{(n)})(\varepsilon^{(n)}+1)
  -2\chi^{(n)}\beta_{1}^{(n)}\eta^{(n)}\\
&\quad-4\beta_{1}^{(n)}(\beta_{2}^{(n)})^2\theta^{(n)}(\varepsilon^{(n)} +1)
-2\beta_{1}^{(n)}\beta_{2}^{(n)}(\beta_{1}^{(n)}-\beta_{2}^{(n)}
)\gamma^{(n)}\Big],\\
  \hat{m}_{ 13}^{(n)}=&-\hat{m}_{31}^{(n)} =\frac{{\rm
i}\mu\alpha_{1}^{(n)}\beta_{2}^{(n)}}{\chi^{(n)}\hat{\chi}^{(n)}}
  \Big[\chi^{(n)}(\beta_{1}^{(n)}-\beta_{2}^{(n)})
+2\beta_{1}^{(n)}(\kappa_{2}^2-2(\beta_{2}^{(n)})^2)\theta^{(n)}\Big ],\\
  \hat{m}_{ 22}^{(n)}=&\frac{{\rm i}\mu}{\chi^{(n)}\hat{\chi}_n^{(n)}}
\Big[\chi^{(n)}[(\alpha_{2}^{(n)})^2(\beta_{1}^{(n)}-\beta_{2}^{(n)})+\beta_{2}^{(n)}\chi^{(n)}]
  (\varepsilon^{(n)}+1)\\
&\quad+4(\alpha_{1}^{(n)})^2\beta_{1}^{(n)}(\beta_{2}^{(n)})^2\theta^{(n)}
(\varepsilon^{(n)}+1)-2(\alpha_{2}^{(n)})^2\beta_{1}^{(n)}\kappa_{2}^2\eta^{(n)}\Big],\\
  \hat{m}_{23}^{(n)}=&-\hat{m}_{32}^{(n)} =\frac{{\rm
i}\mu\alpha_{2}^{(n)}\beta_{2}^{(n)}}{\chi^{(n)}\hat{\chi}^{(n)}}
  \Big[\chi^{(n)}(\beta_{1}^{(n)}-\beta_{2}^{(n)})
+2\beta_{1}^{(n)}(\kappa_{2}^2-2(\beta_{2}^{(n)})^2)\theta^{(n)}\Big
],\\
\hat{m}_{ 33}^{(n)}=&\frac{{\rm
i}\mu\beta_{2}^{(n)}\kappa_{2}^2}{\chi^{(n)} \hat{\chi}^{(n)}}
  \Big[\chi^{(n)}(\varepsilon^{(n)}+1)
  -2\beta_{1}^{(n)}\beta_{2}^{(n)}\eta^{(n)}\Big],
\end{align*}
where
\begin{align*}
  \varepsilon^{(n)}=&
  2e^{{\rm i}\beta_{2}^{(n)}\zeta}/
  (e^{-{\rm i}\beta_{2}^{(n)}\zeta}-e^{{\rm
i}\beta_{2}^{(n)}\zeta}),\\
  \theta^{(n)}=&(e^{{\rm i}\beta_{2}^{(n)}\zeta}-e^{{\rm
i}\beta_{1}^{(n)}\zeta})^2/
  ((1-e^{2{\rm i}\beta_{1}^{(n)}\zeta})(1-e^{2{\rm
i}\beta_{2}^{(n)}\zeta})),\\
  \eta^{(n)}=&(e^{2{\rm i}\beta_{2}^{(n)}\zeta}-e^{2{\rm
i}\beta_{1}^{(n)}\zeta})/
  ((1-e^{2{\rm i}\beta_{1}^{(n)}\zeta})(1-e^{2{\rm
i}\beta_{2}^{(n)}\zeta})),\\
  \gamma^{(n)}=&(e^{2{\rm i}\beta_{1}^{(n)}\zeta}+e^{4{\rm
i}\beta_{2}^{(n)}\zeta})^2/
  ((1-e^{2{\rm i}\beta_{1}^{(n)}\zeta})(1-e^{2{\rm
i}\beta_{2}^{(n)}\zeta})^2),\\
  \hat{\chi}^{(n)}=&\chi^{(n)}+4\big( (\alpha_{1}^{(n)})^2
+(\alpha_{2}^{(n)})^2\big)\beta_{1}^{(n)}\beta_{2}^{(n)}\theta_n/\chi^{(n)} .
\end{align*}
Equivalently, we have the transparent boundary condition for the total field
$\hat{\boldsymbol u}$:
\[
 \mathscr{D}\hat{\boldsymbol u}=\mathscr{T}^{\rm PML}\hat{\boldsymbol
u}+\boldsymbol{g}^{\rm PML} \quad\text{on} ~ \Gamma_h,
\]
where $\boldsymbol{g}^{\rm PML}=\mathscr{D}\hat{\boldsymbol u}_{\rm
inc}-\mathscr{T}^{\rm PML}\hat{\boldsymbol u}_{\rm inc}$.

The PML problem can be reduced to the following boundary value problem:
\begin{equation}\label{cbvp}
 \begin{cases}
  \mu\Delta\boldsymbol{u}^{\rm PML}+(\lambda+\mu)\nabla
\nabla\cdot\boldsymbol{u}^{\rm
PML}+\omega^2\boldsymbol{u}^{\rm PML}=0\quad&\text{in}~ \Omega,\\
\mathscr{D}\boldsymbol{u}^{\rm PML}=\mathscr{T}^{\rm PML}\boldsymbol{u}^{\rm
PML}+\boldsymbol{g}^{\rm PML}\quad&\text{on}~ \Gamma_h,\\
\boldsymbol{u}^{\rm PML}=0 \quad&\text{on}~ \Gamma_f.
 \end{cases}
\end{equation}
The weak formulation of \eqref{cbvp} is to find $\boldsymbol{u}^{\rm PML}\in
H^1_{\rm qp}(\Omega)^3$ such that
\begin{equation}\label{cwp}
 a^{\rm PML}(\boldsymbol{u}^{\rm PML},
\boldsymbol{v})=\langle\boldsymbol{g}^{\rm PML},
\boldsymbol{v}\rangle_{\Gamma_h}\quad\forall\,\boldsymbol{v}\in H^1_{\rm
qp}(\Omega)^3,
\end{equation}
where the sesquilinear form $a^{\rm PML}: H^1_{\rm qp}(\Omega)^3\times
H^1_{\rm qp}(\Omega)^3\to\mathbb{C}$ is defined by
\begin{align}\label{csf}
 a^{\rm PML}(\boldsymbol{u}, \boldsymbol{v})=\mu\int_\Omega
\nabla\boldsymbol{u}:\nabla\bar{\boldsymbol
v}{\rm d}\boldsymbol{x}+(\lambda+\mu)\int_\Omega (\nabla\cdot\boldsymbol{u}
)(\nabla\cdot\bar{\boldsymbol v})\,{\rm d}\boldsymbol{x}\notag\\
-\omega^2\int_\Omega \boldsymbol{u}\cdot\bar{\boldsymbol v}\, {\rm
d}\boldsymbol{x}- \langle\mathscr{T}^{\rm
PML}\boldsymbol{u}, \boldsymbol{v}\rangle_{\Gamma_h}.
\end{align}

The following lemma establishes the relationship between the variational
problem \eqref{cwp} and the weak formulation \eqref{twp}. The proof is
straightforward based on our constructions of the transparent boundary
conditions for the PML problem. The details of the proof is omitted for
simplicity.

\begin{lemm}
Any solution $\hat{\boldsymbol u}$ of the variational problem \eqref{twp}
restricted to $\Omega$ is a solution of the variational \eqref{cwp}; conversely,
any solution $\boldsymbol{u}^{\rm PML}$ of the variational problem \eqref{cwp}
can be uniquely extended to the whole domain to be a solution $\hat{\boldsymbol
u}$ of the variational problem \eqref{twp} in $D$.
\end{lemm}

\subsection{Convergence of the PML solution}

Now we turn to estimating the error between $\boldsymbol{u}^{\rm PML}$ and
$\boldsymbol{u}$. The key is to estimate the error of the boundary operators
$\mathscr{T}^{\rm PML}$ and $\mathscr{T}$.

Let
\[
{\Delta}^-_j=\min\{\Delta_{j}^{(n)}: n\in
U_j\},\quad {\Delta}^+_{j}=\min\{\Delta_{j}^{(n)}: n\notin
U_j\},
\]
where
\[
\Delta_j^{(n)}=|\kappa_j^2-|\boldsymbol{\alpha}^{(n)}|^2|^{1/2},\quad
U_j=\{n: |\boldsymbol{\alpha}^{(n)}|<\kappa_j\}.
\]
Denote
\begin{align*}
K=\frac{48(49+\kappa_{2}^2)^{7/2}}{\kappa_{1}^2}
\times&\max\bigg{\{}\frac{1}{e^{{\Delta}^-_{\,1}{\rm
Im}\zeta} - 1},\, \frac{1}{(e^{\frac{1}{2}{\Delta}^-_{\,1}{\rm
Im}\zeta} - 1)^2},\,
\frac{1}{(e^{\frac{1}{3}{\Delta}^-_{\,1}{\rm Im}\zeta} - 1)^3},\\
&\hspace{1.15cm}\frac{1}{e^{{\Delta}^+_{2}{\rm
Re}\zeta} - 1},\, \frac{1}{(e^{\frac{1}{2}{\Delta}^+_{2}{\rm
Re}\zeta} - 1)^2},\, \frac{1}{(e^{\frac{1}{3}{\Delta}^+_{2}{\rm
Re}\zeta} - 1)^3},\\
&\hspace{1.15cm}\frac{1}{e^{{\Delta}^-_{\,2}{\rm Im}\zeta} - 1},\,
\frac{(e^{-{\Delta}^+_{1}{\rm Re}\zeta}
     +e^{-{\Delta}^-_{\,2}{\rm Im}\zeta})^2}
     {(1-e^{-2{\Delta}^+_{1}{\rm Im}\zeta})
      (1-e^{-2{\Delta}^-_{\,2}{\rm Re}\zeta})^2},\,
\bigg{\}}.
\end{align*}
The constant $K$ can be used to control the modeling error between the
PML problem and the original scattering problem. Once the incoming plane wave
$\boldsymbol{u}^{\rm inc}$ is fixed, the quantities ${\Delta}^-_j,
{\Delta}^+_{j}$ are fixed. Thus the constant $K$ approaches to zero
exponentially as the PML parameters ${\rm Re}\zeta$ and ${\rm Im}\zeta$ tend to
infinity. Recalling the definition of $\zeta$ in \eqref{zeta}, we know that
${\rm Re}\zeta$ and ${\rm Im}\zeta$ can be calculated by the medium property
$\rho(x_3)$, which is usually taken as a power function:
\[
 \rho(x_3)=1+\sigma\left(\dfrac{x_3-b}{\delta}\right)^{m} \quad\text{if}~ x_3\geq b,\quad m\geq 1.
\]
Thus we have
\[
 {\rm Re}\zeta=\left(1+\frac{{\rm Re}\sigma}{m+1}\right)\delta, \quad
 {\rm Im}\zeta=\left(\frac{{\rm Im}\sigma}{m+1}\right)\delta.
\]
In practice, we may pick some appropriate PML parameters $\sigma$ and $\delta$
such that ${\rm Re}\zeta\geq 1$.

\begin{lemm}\label{boe}
For any $\boldsymbol{u}, \boldsymbol{v}\in H^1_{\rm qp}(\Omega)^3$, we have
\[
|\langle (\mathscr{T}^{\rm PML}-\mathscr{T})\boldsymbol{u},
\boldsymbol{v}\rangle_{\Gamma_h}|\leq
\hat{K}\|\boldsymbol{u}\|_{L^2(\Gamma_h)^3}
\|\boldsymbol{v}\|_{L^2(\Gamma_h)^3},
\]
where $\hat{K}=11\mu^2 K/\kappa_{1}^4$.
\end{lemm}

\begin{proof}
For any $\boldsymbol{u}, \boldsymbol{v}\in H^1_{\rm qp}(\Omega)^3$, we have
the following Fourier series expansions:
\[
 \boldsymbol{u}(\boldsymbol r,h)=\sum_{n\in\mathbb{Z}^2}\boldsymbol{u}^{(n)}
(h)  e^{{\rm i}\boldsymbol{\alpha}^{(n)}\cdot\boldsymbol{r}},\quad
\boldsymbol{v}(\boldsymbol r, h)=\sum_{n\in\mathbb{Z}^2}\boldsymbol{v}^{(n)}(h)
 e^{{\rm i}\boldsymbol{\alpha}^{(n)}\cdot\boldsymbol{r}},
\]
which gives
\[
 \|\boldsymbol{u}\|^2_{L^2(\Gamma_h)^3}=\Lambda_1\Lambda_2\sum_{n\in\mathbb{Z
}^2}|\boldsymbol {u}^{(n)} (h)|^2 , \quad
\|\boldsymbol{v}\|^2_{L^2(\Gamma_h)^3}=\Lambda_1\Lambda_2\sum_{
n\in\mathbb{Z}^2}|\boldsymbol{v}^{(n)} (h)|^2.
\]
It follows from the orthogonality of Fourier series, the Cauchy--Schwarz
inequality, and Proposition \ref{mfe} that we have
\begin{align*}
&|\langle (\mathscr{T}^{\rm PML}-\mathscr{T})\boldsymbol{u},
\boldsymbol{v}\rangle_{\Gamma_h}|
=\bigg{|}\Lambda_1\Lambda_2\sum_{n\in\mathbb{Z}^2}\bigl(
(M^{(n)}-\hat{M}^{(n)})\boldsymbol{u}^{(n)}(h)\bigr)\cdot\bar{
\boldsymbol v}^{(n)}(h)\bigg{|}\\
&\leq\Bigl(\Lambda_1\Lambda_2\sum_{n\in\mathbb{Z}^2}\|M^{(n)}-\hat{M}^{(n)}
\|_F^2\, |{\boldsymbol u}^{(n)}(h)|^2\Bigr)^{1/2}
\Bigl(\Lambda_1\Lambda_2\sum_{n\in\mathbb{Z}^2}|
\boldsymbol{v}^{(n)}(h)|^2\Bigr)^{1/2}\\
&\leq\hat{K}\|\boldsymbol{u}\|_{L^2(\Gamma_h)^3}
\|\boldsymbol{v}\|_{L^2(\Gamma_h)^3},
\end{align*}
which completes the proof.
\end{proof}

Let $a=\min\{f(\boldsymbol x): \boldsymbol{x}\in \Gamma_f\}$. Denote
$\tilde\Omega=\{\boldsymbol{x}\in\mathbb{R}^3: 0<x_1<\Lambda_1,\,
0<x_2<\Lambda_2,\, a<x_3<h\}.$
\begin{lemm}\label{tr}
 For any $\boldsymbol{u}\in H^1_{\rm qp}(\Omega)^3$, we have
 \[
\|\boldsymbol{u}\|_{L^2(\Gamma_h)^3}
\leq \|\boldsymbol{u}\|_{H^{1/2}(\Gamma_h)^3}
\leq\gamma_2\|\boldsymbol{u}\|_{H^1(\Omega)^3},
 \]
where $\gamma_2=(1+(h-a)^{-1})^{1/2}$.
\end{lemm}

\begin{proof}
A simple calculation yields
\begin{align*}
 (h-a)|u(h)|^2
 &=\int_{a}^{h}|u(x_3)|^2{\rm d}x_3
   +\int_{a}^{h}\int_{x_3}^{h}
   \frac{{\rm d}}{{\rm d} t}|u(t)|^2{\rm d}t{\rm d}{x_3}\\
&\leq \int_{a}^{h}|u(x_3)|^2{\rm d}{x_3}
+(h-a)\int_{a}^{h} 2|u(t)||u'(t)|{\rm d}t,
\end{align*}
which gives by applying the Young's inequality that
\[
 (1+|\boldsymbol{\alpha}^{(n)}|^2)^{1/2}|u(h)|^2
 \leq \gamma_2^2 (1+|\boldsymbol{\alpha}^{(n)}|^2)
 \int_{a}^{h} |u(t)|^2{\rm d}t+\int_{a}^{h} |u'(t)|^2{\rm d}t.
\]
Given $\boldsymbol{u}\in H^1_{\rm qp}(\Omega)^3$, we consider the zero extension
\[
\tilde{\boldsymbol{u}}=
\begin{cases}
\boldsymbol{u} \quad&\text{in}~ \Omega,\\
0    \quad&\text{in}~ \tilde\Omega\backslash\bar\Omega,
 \end{cases}
\]
which has the Fourier series expansion
\[
\tilde{\boldsymbol{u}}(\boldsymbol x)=\sum_{n\in\mathbb{Z}^2}\tilde{\boldsymbol{
u}}^{(n)}(x_3)e^{{\rm i}\boldsymbol{\alpha}^{(n)}\cdot\boldsymbol{r}}
\quad\text{in} ~ \tilde\Omega.
\]
By definitions, we have
\[
 \|\tilde{\boldsymbol {u}}\|^2_{H^{1/2}(\Gamma_h)^3}
 =\Lambda_1\Lambda_2\sum_{n\in\mathbb{Z}^2}
(1+|\boldsymbol{\alpha}^{(n)}|^2)^{1/2}|\tilde{\boldsymbol {u}}^{(n)} (h)|^2
\]
and
\[
 \|\tilde{\boldsymbol{u}}\|^2_{H^1(\Omega)^3}
 =\Lambda_1\Lambda_2\sum_{n\in\mathbb{Z}^2} \int_{a}^{h}
(1+|\boldsymbol{\alpha}^{(n)}|^2)|\tilde{\boldsymbol
{u}}^{(n)}(x_3)|^2+|\boldsymbol{u}^{(n)'} (x_3)|^2 {\rm d}x_3.
\]
The proof is completed by combining the above estimates and noting
$\|\boldsymbol{u}\|^2_{H^{1/2}(\Gamma_h)^3}=\|\tilde{\boldsymbol{u}}\|^2_{H^{1/2
} (\Gamma_h)^3}$ and
$\|\boldsymbol{u}\|^2_{H^1(\Omega)^3}=\|\tilde{\boldsymbol{u}}\|^2_{
H^1(\tilde\Omega)^3}$.
\end{proof}

\begin{theo}
 Let $\gamma_1$ and $\gamma_2$ be the constants in the inf-sup condition
\eqref{ifc} and in Lemma \ref{tr}, respectively. If
$\hat{K}\gamma^2_2<\gamma_1$, then the PML variational problem
\eqref{cwp} has a unique weak solution $\boldsymbol{u}^{\rm PML}$, which
satisfies the error estimate
\begin{align}\label{ee}
 \|\boldsymbol{u}-\boldsymbol{u}^{\rm PML}\|_\Omega:=\sup_{0\neq
\boldsymbol{v}\in H^1_{\rm qp}(\Omega)^3}
\frac{|a(\boldsymbol{u}-\boldsymbol{u}^{\rm PML},
\boldsymbol{v})|}{\|\boldsymbol{v}\|_{H^1(\Omega)^3}}
\leq \gamma_2\hat{K} \|\boldsymbol{u}^{\rm PML}-\boldsymbol{u}^{\rm
inc}\|_{L^2(\Gamma_h)^3},
\end{align}
where $\boldsymbol{u}$ is the unique weak solution of the variational problem
\eqref{wp}.
\end{theo}

\begin{proof}
It suffices to show the coercivity of the sesquilinear form $a^{\rm PML}$
defined in \eqref{csf} in order to prove the unique solvability of the weak
problem \eqref{cwp}. Using Lemmas \ref{boe}, \ref{tr} and the assumption
$\hat{K}\gamma^2_2<\gamma_1$, we get for any $\boldsymbol{u},
\boldsymbol{v}$ in $H^1_{\rm qp}(\Omega)^3$ that
\begin{align*}
 |a^{\rm PML}(\boldsymbol{u}, \boldsymbol{v})|&\geq |a(\boldsymbol{u},
\boldsymbol{v})|-\langle (\mathscr{T}^{\rm
PML}-\mathscr{T})\boldsymbol{u},
\boldsymbol{v}\rangle_{\Gamma_h}|\\
&\geq|a(\boldsymbol{u},
\boldsymbol{v})|-\hat{K}\gamma_2^2\|\boldsymbol{u}\|_{
H^1(\Omega)^3} \|\boldsymbol{v}\|_{H^1(\Omega)^3}\\
&\geq\bigl(\gamma_1-\hat{K}\gamma_2^2\bigr)\|\boldsymbol{u}\|_{
H^1(\Omega)^3}\|\boldsymbol{v}\|_{H^1(\Omega)^3}.
\end{align*}
It remains to show the error estimate \eqref{ee}. It follows from
\eqref{wp}--\eqref{sf} and \eqref{cwp}--\eqref{csf} that
\begin{align*}
 a(\boldsymbol{u}-\boldsymbol{u}^{\rm PML},
\boldsymbol{v})&=a(\boldsymbol{u}, \boldsymbol{v})-a(\boldsymbol{u}^{\rm PML},
\boldsymbol{v})\\
&=\langle\boldsymbol{f},
\boldsymbol{v}\rangle_{\Gamma_h}-\langle\boldsymbol{f}^{\rm PML},
\boldsymbol{v}\rangle_{\Gamma_h}+a^{\rm PML}(\boldsymbol{u}^{\rm PML},
\boldsymbol{v})-a(\boldsymbol{u}^{\rm PML}, \boldsymbol{v})\\
&=\langle (\mathscr{T}^{\rm PML}-\mathscr{T})\boldsymbol{u}^{\rm
inc}, \boldsymbol{v}\rangle_{\Gamma_h} -\langle (\mathscr{T}^{\rm
PML}-\mathscr{T})\boldsymbol{u}^{\rm PML},
\boldsymbol{v}\rangle_{\Gamma_h}\\
&=\langle (\mathscr{T}-\mathscr{T}^{\rm PML})(\boldsymbol{u}^{\rm
PML}-\boldsymbol{u}^{\rm inc}), \boldsymbol{v}\rangle_{\Gamma_h},
\end{align*}
which completes the proof upon using Lemmas \ref{boe} and \ref{tr}.
\end{proof}

We remark that the PML approximation error can be
reduced exponentially by either enlarging the thickness $\delta$ of the PML
layers or enlarging the medium parameters ${\rm Re}\sigma$ and ${\rm
Im}\sigma$.

\section{Numerical experiments}

In this section, we present two examples to demonstrate the numerical
performance of the PML solution. The first-order linear element is used for
solving the problem. Our implementation is based on parallel hierarchical
grid (PHG) \cite{phg}, which is a toolbox for developing parallel adaptive
finite element programs on unstructured tetrahedral meshes. The linear system
resulted from finite element discretization is solved by the Supernodal LU
(SuperLU) direct solver, which is a general purpose library for the direct
solution of large, sparse, nonsymmetric systems of linear equations.

{\em Example 1}. We consider the simplest periodic structure, a straight line,
where the exact solution is available. We assume that a plane compressional
plane wave $\boldsymbol u^{\rm inc}=\boldsymbol q e^{{\rm
i}(\boldsymbol\alpha\cdot \boldsymbol r-\beta x_3)}$ is incident on the
straight line $x_3=0$, where $\boldsymbol\alpha=(\alpha_1, \alpha_2)^\top,
\alpha_1=\kappa_1\sin\theta_1\cos\theta_2,
\alpha_2=\kappa_1\sin\theta_1\sin\theta_2, \beta=\kappa_1\cos\theta_1,
\boldsymbol q=(q_1, q_2, q_3), q_1=\sin\theta_1\cos\theta_2,
q_2=\sin\theta_1\sin\theta_2, q_3=-\cos\theta_1, \theta_1\in [0, \pi/2),
\theta_2\in[0, 2\pi]$ are incident angles. It follows from the Navier equation
and the Helmholtz decomposition that we obtain the exact solution:
\[
 \boldsymbol u(\boldsymbol x)=\boldsymbol u^{\rm inc}(\boldsymbol x)+{\rm i}
\begin{bmatrix}
                                               \alpha_1\\
                                               \alpha_2\\
                                               \beta
                                              \end{bmatrix}
  a e^{{\rm i}(\boldsymbol\alpha\cdot\boldsymbol r+\beta x_3)}+{\rm i}
\begin{bmatrix}
\alpha_2 b_3 -
\beta_2^{(0)}b_2\\
\beta_2^{(0)}b_1-\alpha_1 b_3\\
\alpha_1 b_2-\alpha_2 b_1
\end{bmatrix}
 e^{{\rm i}(\boldsymbol\alpha\cdot\boldsymbol r+\beta_2^{(0)}x_3)},
\]
where $(a, b_1, b_2, b_3)$ is the solution of the following linear system:
\[
 {\rm i}
 \begin{bmatrix}
 \alpha_{1} &                        0 & -\beta^{(0)}_{2} &
\alpha_{2}\\[5pt]
 \alpha_{2} & \beta^{(0)}_{2} &         0 & -\alpha_{1}\\[5pt]
 \beta &    -\alpha_{2} &  \alpha_{1} &
0\\[5pt]
 0 &          \alpha_{1} &  \alpha_{2} & \beta^{(0)}_{2}
 \end{bmatrix}
 \begin{bmatrix}
  a\\[5pt]
  b_1\\[5pt]
  b_2\\[5pt]
  b_3
 \end{bmatrix}=-
 \begin{bmatrix}
  q_1\\[5pt]
  q_2\\[5pt]
  q_3\\[5pt]
  0
 \end{bmatrix}.
 \]
 Solving the above equations via Cramer's rule gives
  \begin{align*}
a&=\frac{\rm i}{\chi}\big(\alpha_{1}   q_1
+ \alpha_{2}  q_2+ \beta^{(0)}_{2} q_3\big)\\
b_{1}  &=\frac{\rm i}{\chi}\big(
 \alpha_{1}\alpha_{2}(\beta-\beta^{(0)}_{2}) q_1 /\kappa^2_{2}
 +\big[(\alpha_{1})^2\beta^{(0)}_{2}+(\alpha_{2})^2\beta + \beta(\beta^{(0)}_{2})^2\big] q_2/\kappa^2_{2}
 -\alpha_{2} q_3\big) \\
b_{2}  &=\frac{\rm i}{\chi}\big(
 -\big[(\alpha_{1})^2\beta+(\alpha_{2})^2\beta^{(0)}_{2} + \beta(\beta^{(0)}_{2})^2\big] q_1/\kappa^2_{2}
  -\alpha_{1}\alpha_{2}(\beta-\beta^{(0)}_{2})q_2/\kappa^2_{2} +\alpha_{1} q_3\big) \\
b_{3}  &=\frac{\rm i}{\kappa^2_{2} }\big(
 \alpha_{2} q_1 -\alpha_{1} q_2\big),
 \end{align*}
 where
 \[
 \chi=(|\boldsymbol{\alpha}|^2+\beta\beta^{(0)}_{2}).
 \]
In our experiment, the parameters are chosen as $\lambda=1,
\mu=2, \theta_1=\theta_2=\pi/6, \omega=2\pi$. The computational domain
$\Omega=(0,1)\times(0,1)\times(0,0.6)$ and the PML domain is $\Omega^{\rm
PML}=(0, 1)\times (0, 1)\times (0.3, 0.6)$, i.e., the thickness of the PML layer
is 0.3. We choose $\sigma=25.39$ and $m=2$ for the medium property to ensure the
constant $K$ is so small that the PML error is negligible compared to the
finite element error. The mesh and surface plots of the amplitude of the field
$\boldsymbol{v}_h^{\rm PML}$ are shown in Figure \ref{ex1:mesh}. The mesh has
57600 tetrahedrons and the total number of degrees of freedom (DoFs) on the mesh
is 60000. The grating efficiencies are displayed in Figure \ref{ex1:eff}, which
verifies the conservation of the energy in Theorem \ref{coe}. Figure
\ref{ex1:err} shows the curves of $\log N_k$ versus $\log
\|\boldsymbol{u} -\boldsymbol{u}_k\|_{0, \Omega}$, i.e., $L^2$-error, and
$\log \|\nabla(\boldsymbol{u} -\boldsymbol{u}_k)\|_{0, \Omega}$,
i.e., $H^1$-error, where $N_k$ is the total number of DoFs of the mesh. It
indicates that the meshes and the associated numerical complexity are
quasi-optimal: $\log\|\boldsymbol{u} -\boldsymbol{u}_k\|_{0,
\Omega}=O(N^{-2/3}_k)$ and $\log \|\nabla(\boldsymbol{u}
-\boldsymbol{u}_k)\|_{0, \Omega}=O(N^{-1/3}_k)$ are valid asymptotically.

{\em Example 2}. This example concerns the scattering of the time-harmonic
compressional plane wave $\boldsymbol{u}^{\rm inc}$ on a flat grating surface
with two square bumps, as seen in Figure \ref{ex2:geo}. The parameters are
chosen as $\lambda=1, \mu=2, \theta_1=\theta_2=\pi/6$, $\omega=2\pi$. The
computational domain is $\Omega=(0,1)\times(0,1)\times(0,1)$ and the PML domain
is $\Omega^{\rm PML}=(0, 1)\times(0, 1)\times(0.5, 1.0)$, i.e., the thickness
of the PML layer is 0.5. Again, we choose $\sigma=28.57$ and $m=2$ for the
medium property to ensure that the PML error is negligible compared to the
finite element error. Since there is no analytical solution for this example, we
plot the grating efficiencies against the DoFs in Figure \ref{ex2:eff} to verify
the conservation of the energy. Figure \ref{ex2:mesh} shows the mesh and the
amplitude of the associated solution for the scattered field $\boldsymbol
v_h^{\rm PML}$ when the mesh has 49968 nodes.  

\begin{figure}
\centering
\includegraphics[width=0.38\textwidth]{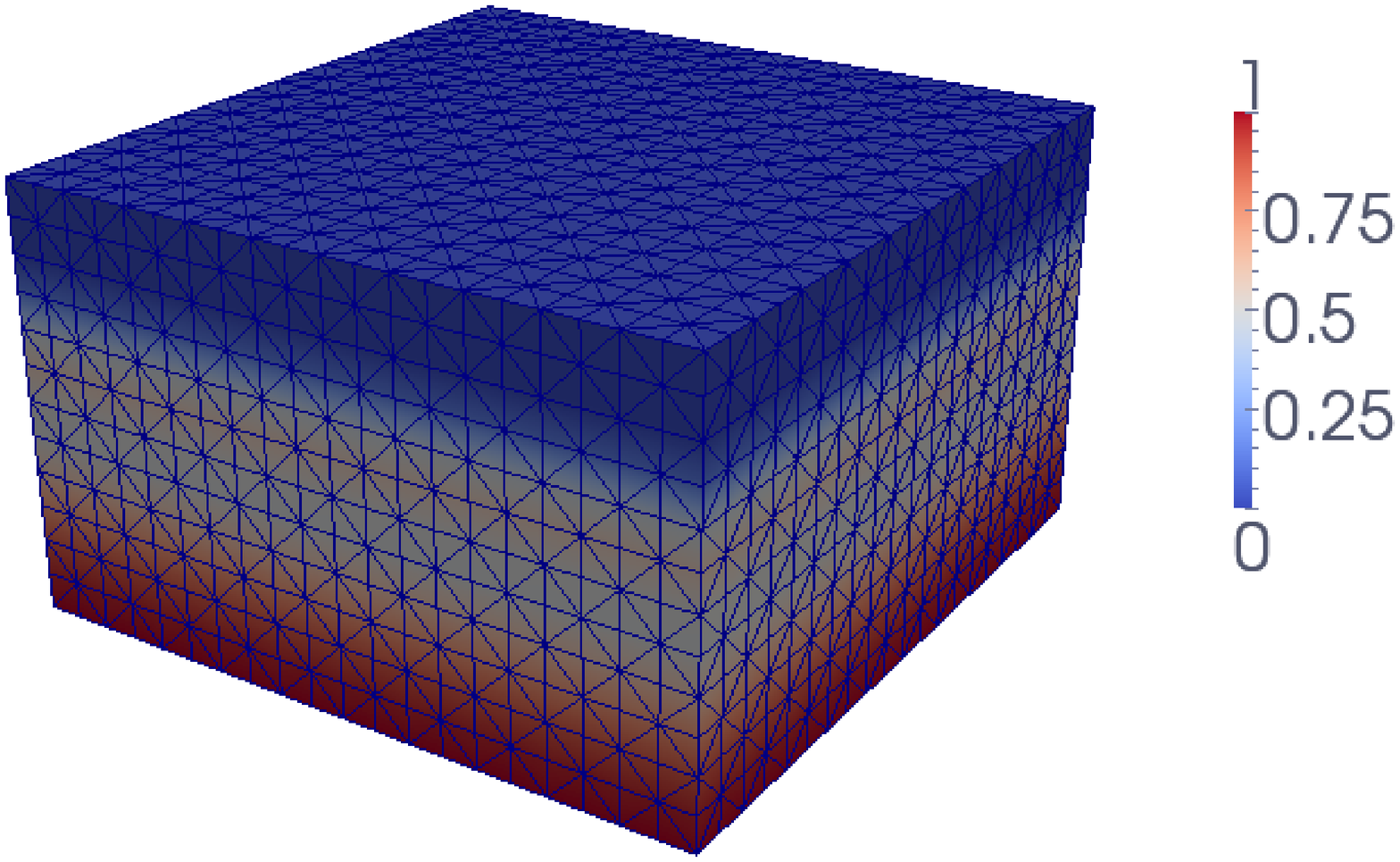}
\hskip 1.5cm
\includegraphics[width=0.38\textwidth]{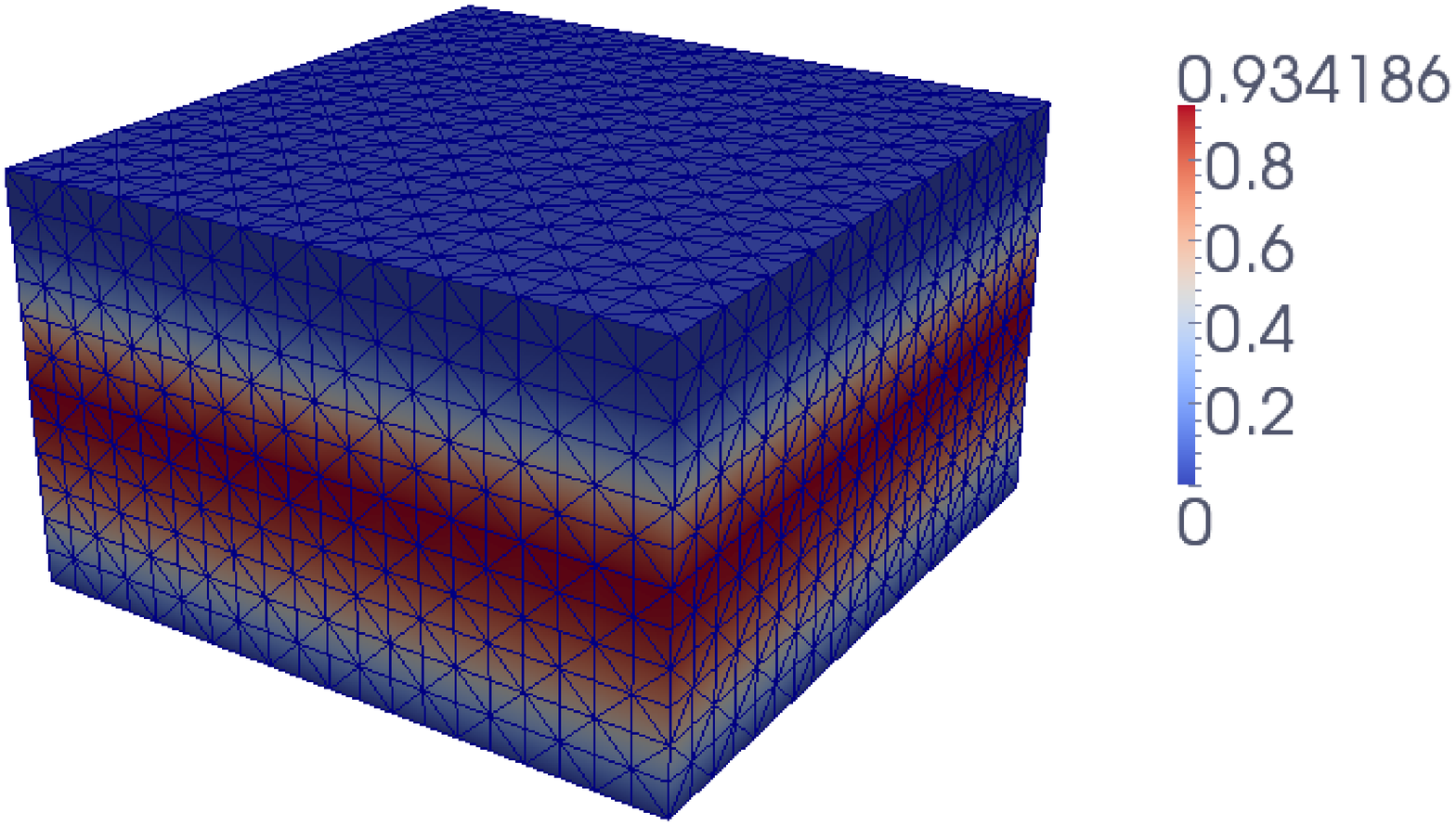}
\caption{The mesh and surface plots of the amplitude of the
associated solution for the scattered field $\boldsymbol v_h^{\rm
PML}$ for Example 1: (left) the amplitude of the real part of the solution
$|{\rm Re}\boldsymbol v_h^{\rm PML}|$; (right) the amplitude of the imaginary
part of the solution $|{\rm Im}\boldsymbol v_h^{\rm PML}|$.}\label{ex1:mesh}
\end{figure}

\begin{figure}
\centering
\includegraphics[width=0.4\textwidth]{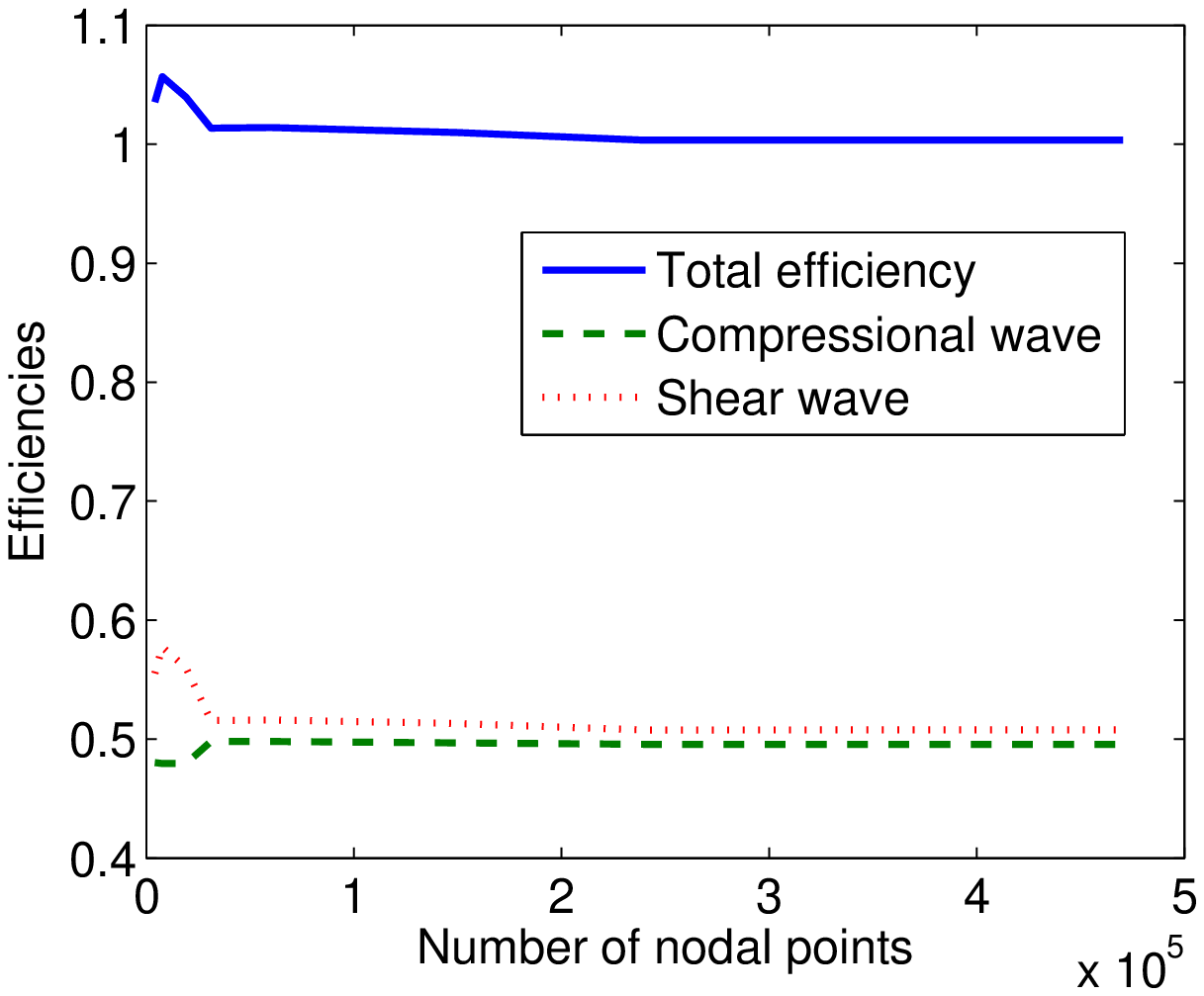}
\hskip 1.5cm
\includegraphics[width=0.4\textwidth]{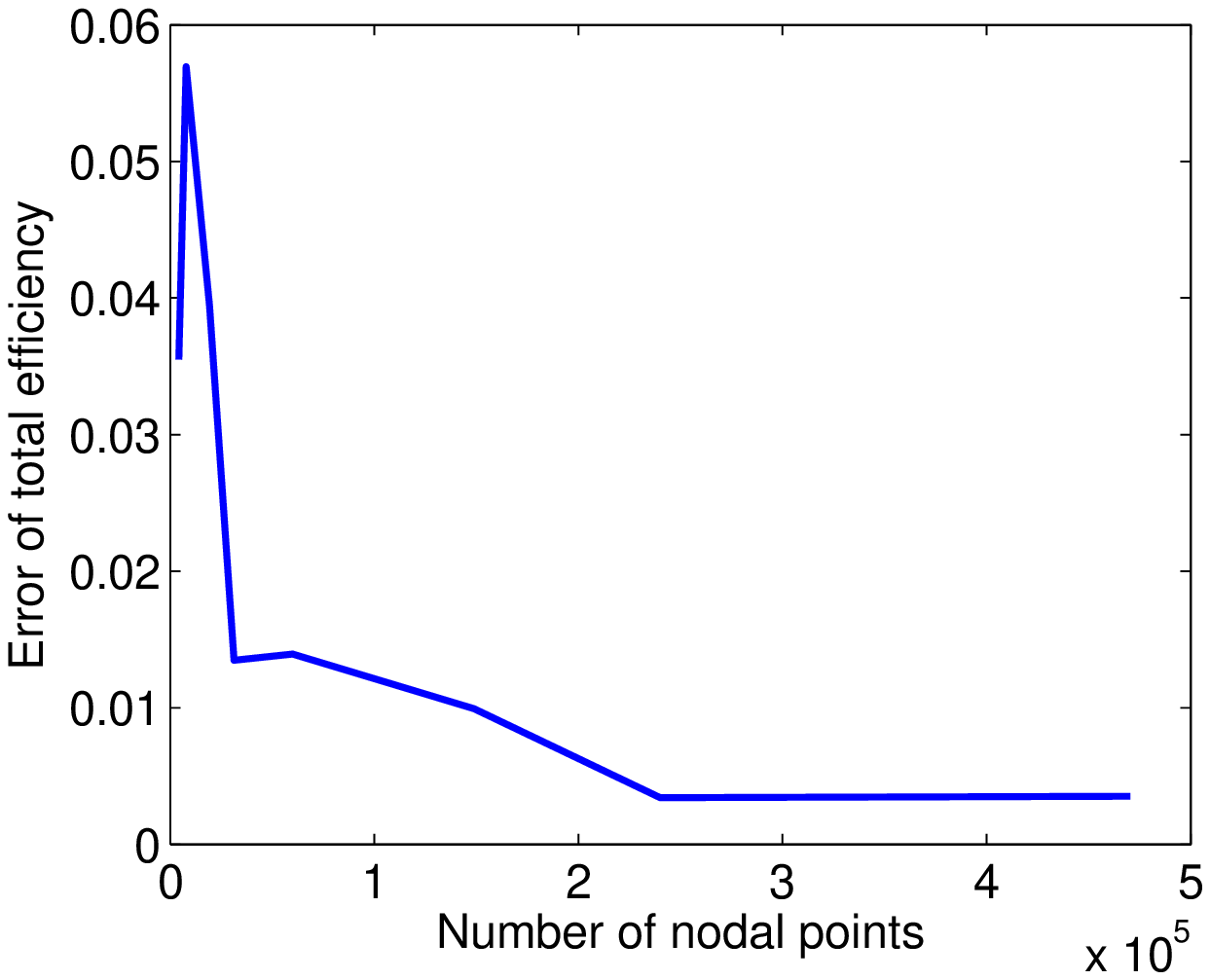}
\caption{Grating efficiencies and robustness of grating efficiency for Example
1.}\label{ex1:eff}
\end{figure}

\begin{figure}
\centering
\includegraphics[width=0.38\textwidth]{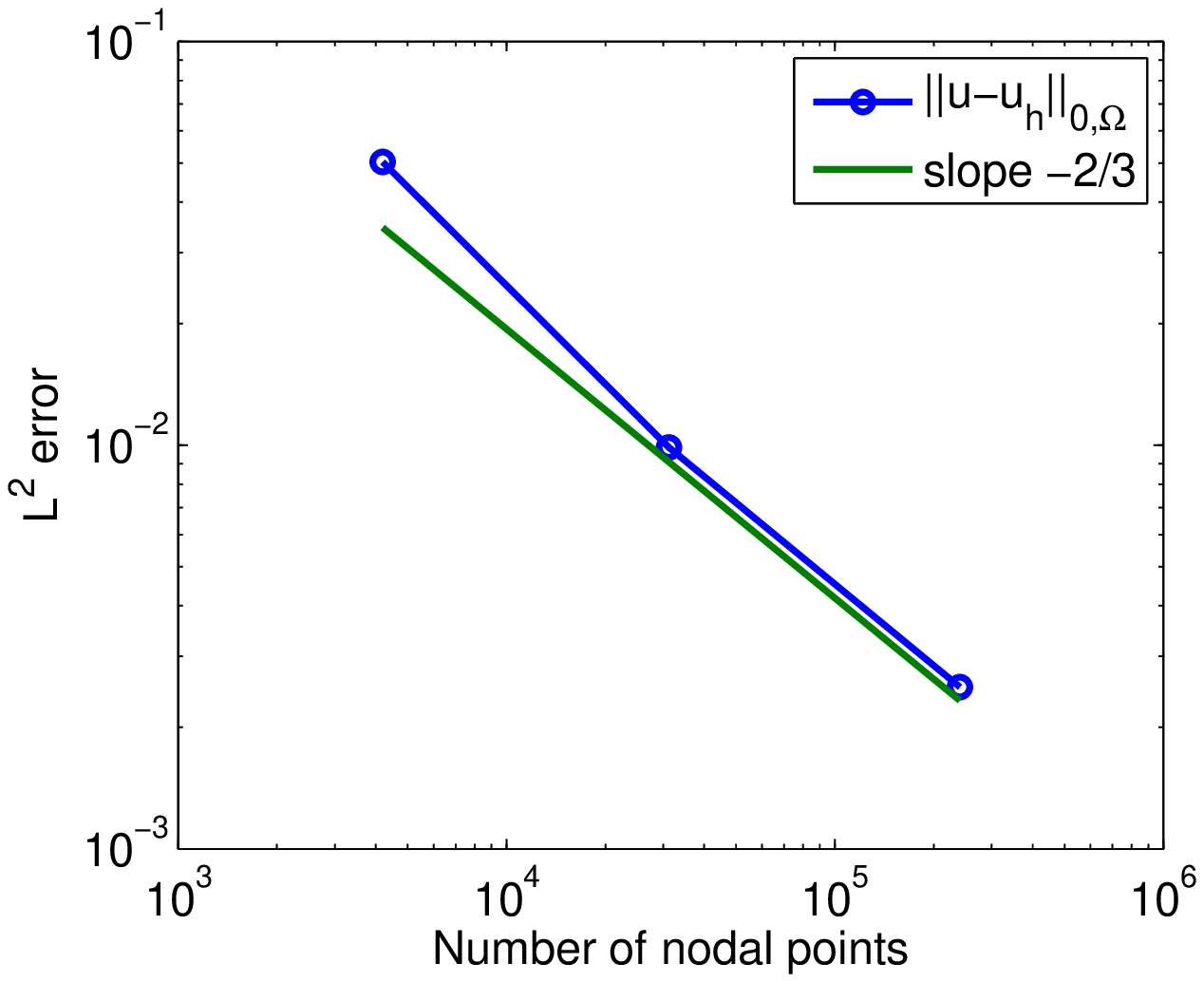}
\hskip 1.5cm
\includegraphics[width=0.38\textwidth]{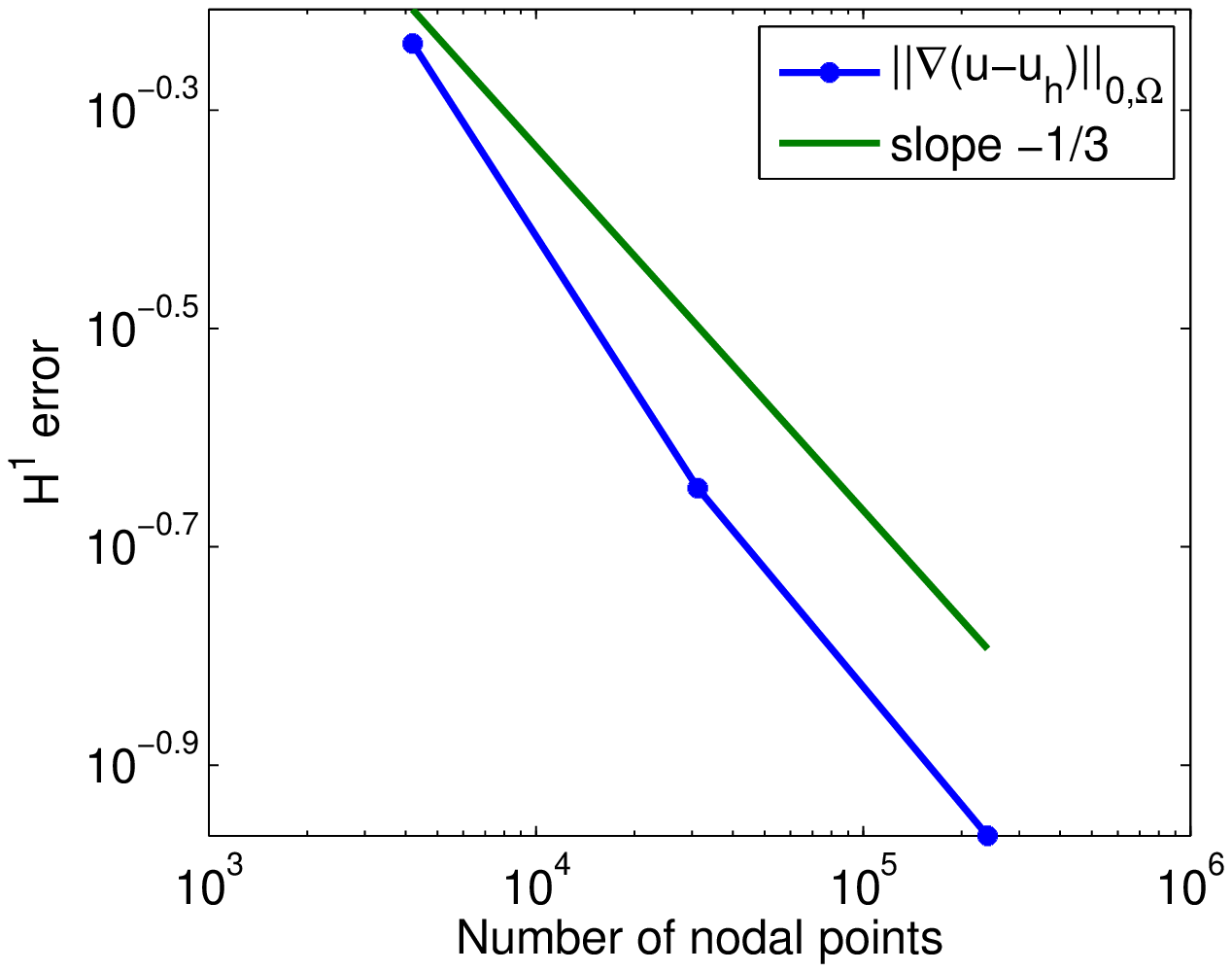}
\caption{Quasi-optimality of $L^2$- and $H^1$- error estimates for Example
1.}\label{ex1:err}
\end{figure}

\begin{figure}
\center
\includegraphics[width=0.5\textwidth]{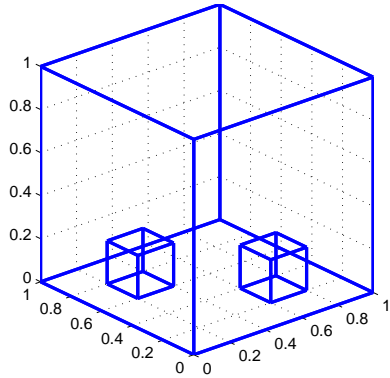}
\caption{Geometry of the domain for Example 2.}
\label{ex2:geo}
\end{figure}

\begin{figure}
\centering
\includegraphics[width=0.4\textwidth]{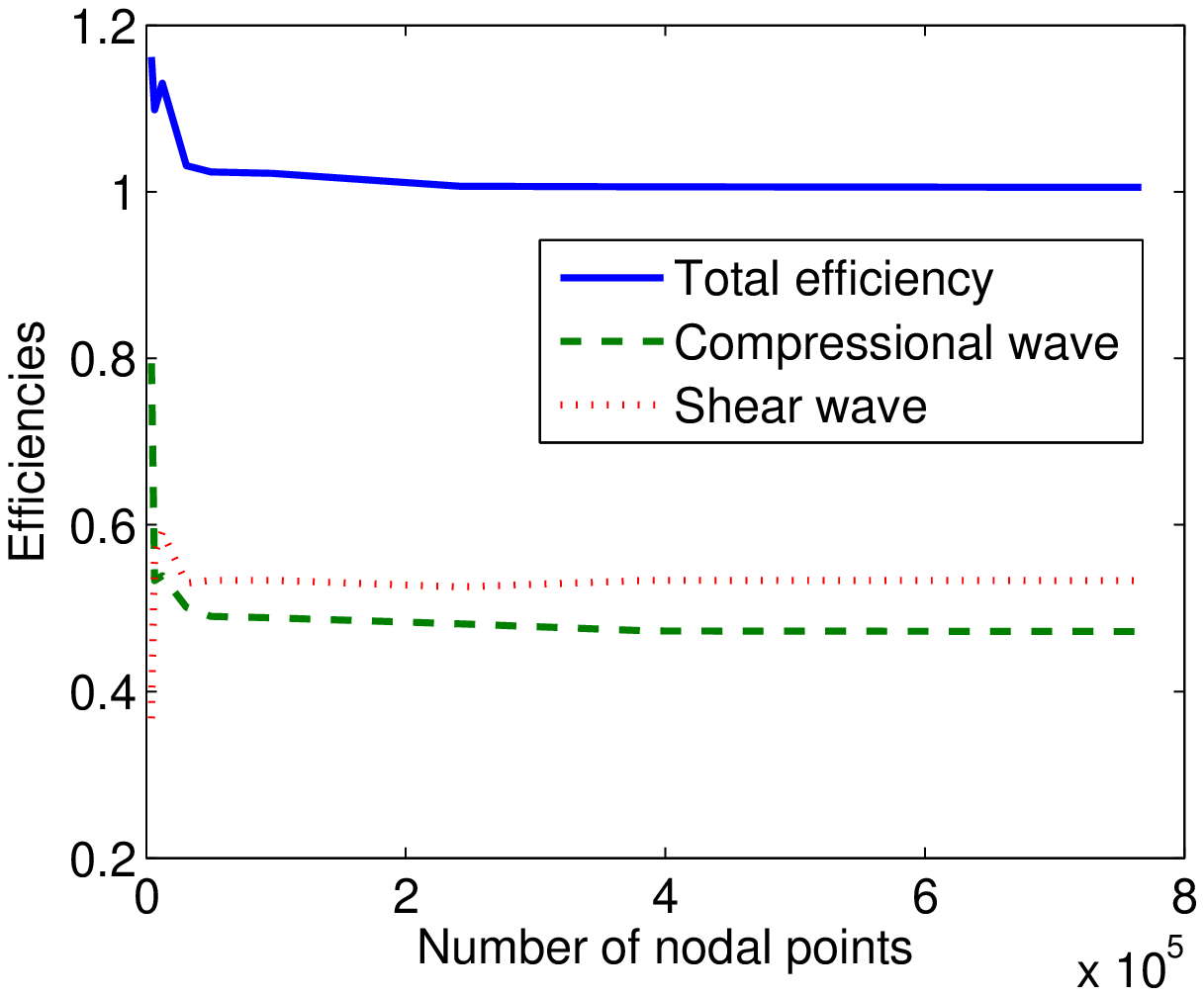}
\hskip 1.5cm
\includegraphics[width=0.4\textwidth]{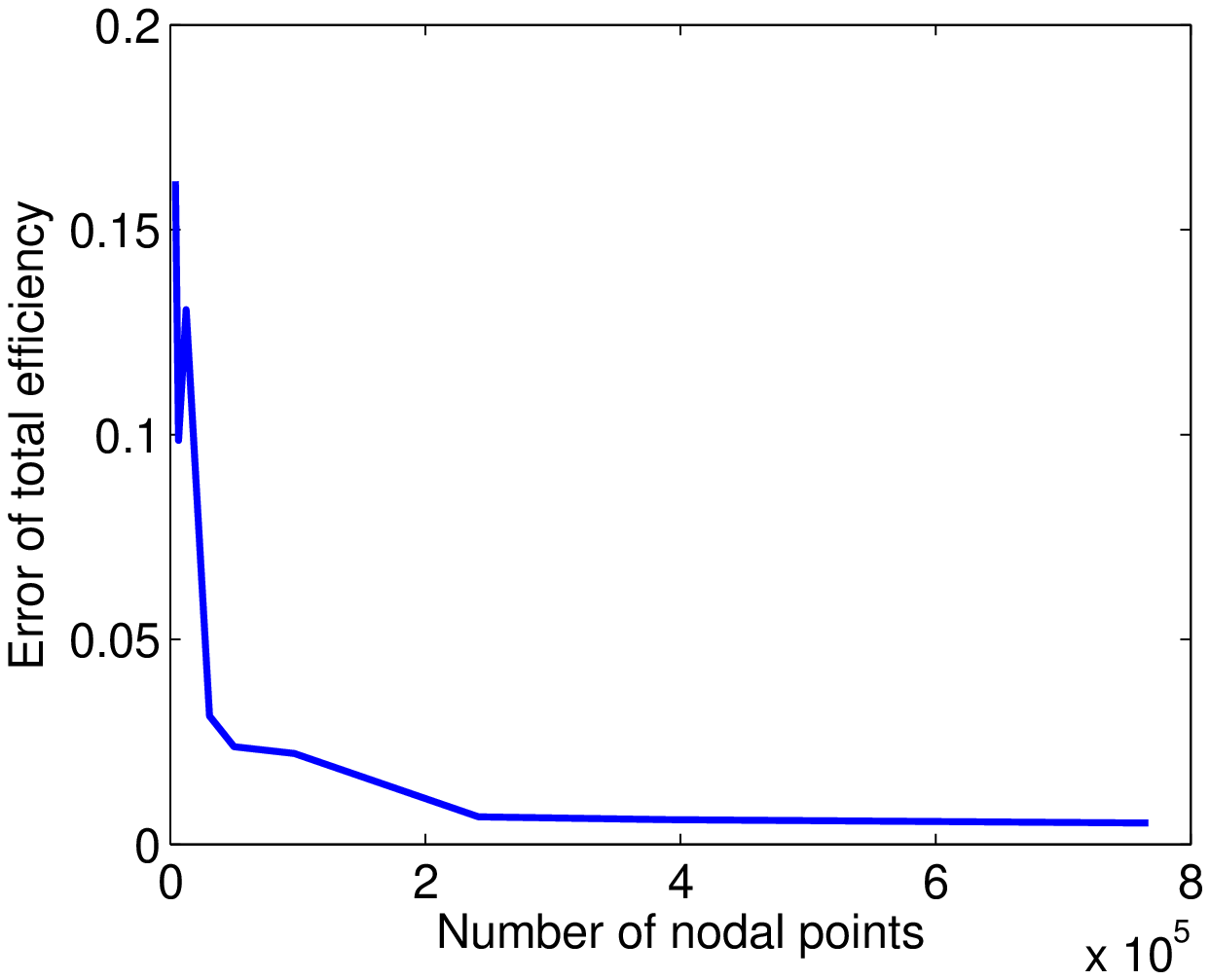}
\caption{Grating efficiencies and robustness of grating efficiency for Example
2.}\label{ex2:eff}
\end{figure}

\begin{figure}
\centering
\includegraphics[width=0.38\textwidth]{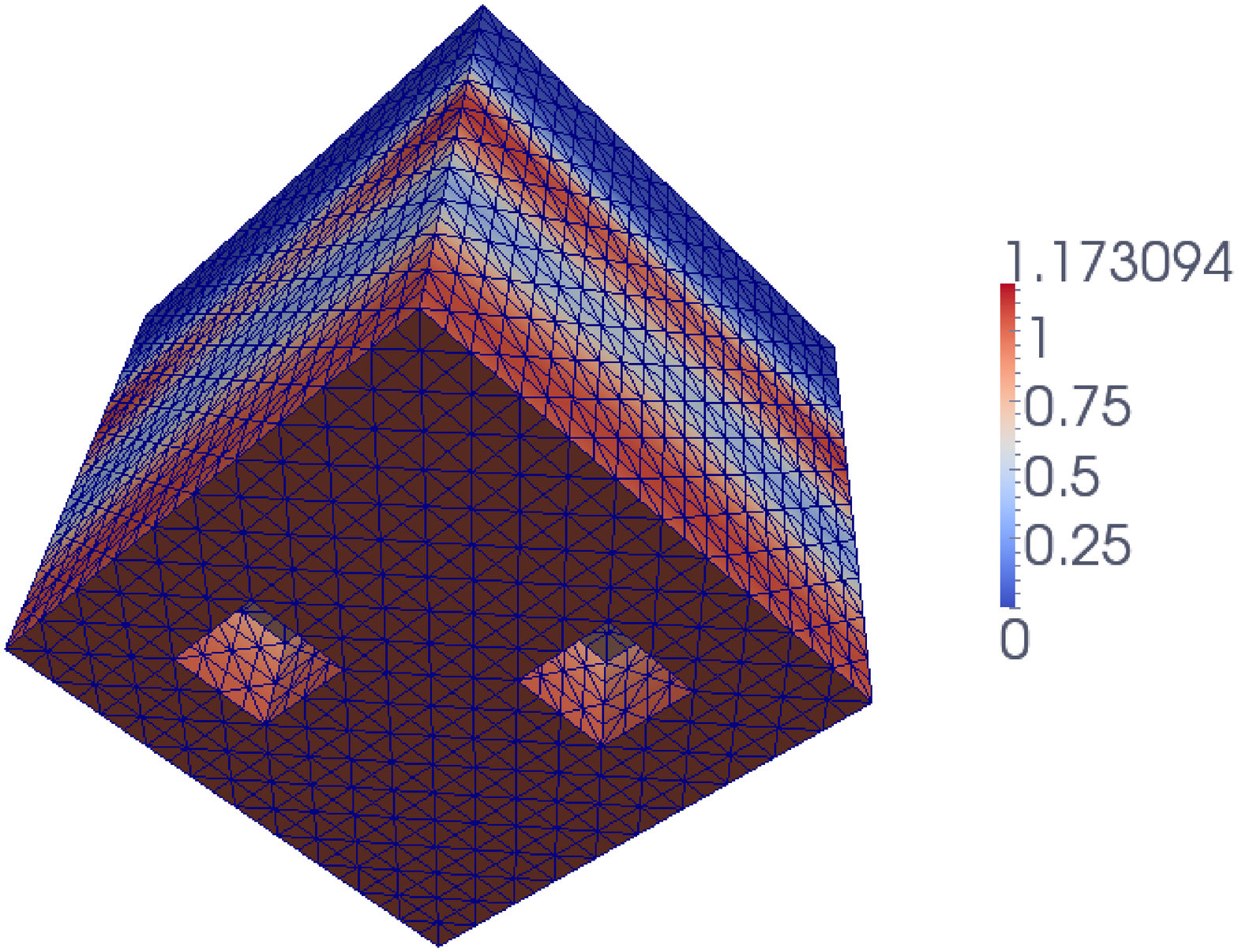}
\hskip 1.5cm
\includegraphics[width=0.38\textwidth]{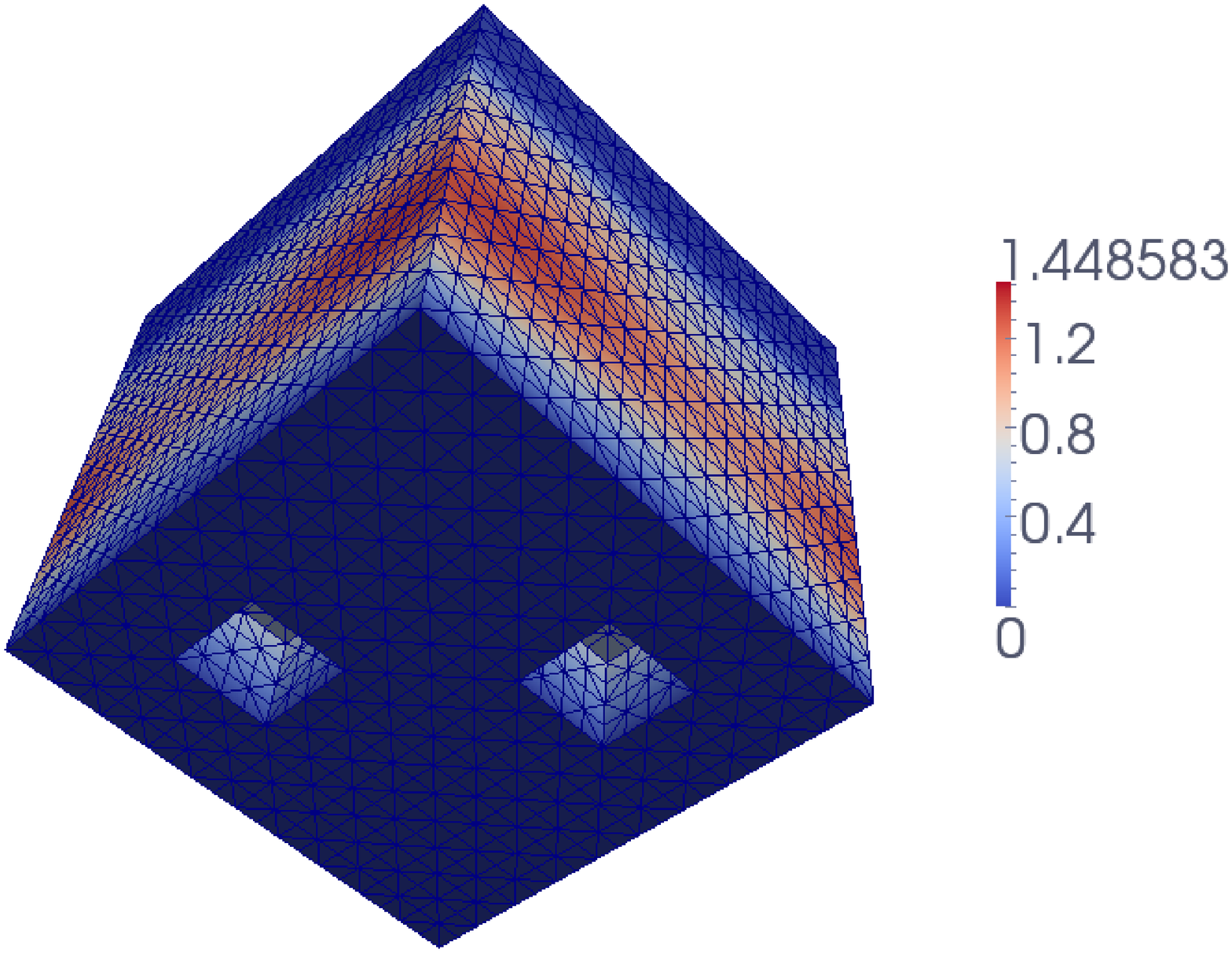}
\caption{The mesh and surface plots of the amplitude of the associated solution
for the scattered field $\boldsymbol v_h^{\rm PML}$ for Example 2:
(left) the amplitude of the real part of the solution $|{\rm
Re}\boldsymbol v_h^{\rm PML}|$; (right) the amplitude of the
imaginary part of the solution $|{\rm Im}\boldsymbol
v_h^{\rm PML}|$.}\label{ex2:mesh}
\end{figure}

\section{Concluding remarks}

We have studied a variational formulation for the elastic wave scattering
problem in a biperiodic structure and adopted the PML to truncate the
physical domain. The scattering problem is reduced to a boundary value problem
by using transparent boundary conditions. We prove that the truncated PML
problem has a unique weak solution which converges exponentially to the solution
of the original problem by increasing the PML paramers. Numerical results show
that the proposed method is effective to solve the scattering problem of elastic
waves in biperiodic structures. Although the paper presents the results for the
rigid boundary condition, the method is applicable to other boundary
conditions or the transmission problem where the structures are
penetrable. This work considers only the uniform mesh refinement. We plan to
incorporate the adaptive mesh refinement with a posteriori error estimate for
the finite element method to handle the problems where the solutions may have
singularities. The progress will be reported elsewhere in a future work. 

\appendix

\section{Technical estimates}

In this section, we present the proofs for some technical estimates which are
used in our analysis for the error estimate between the solutions of the PML
problem and the original scattering problem.

\begin{prop}\label{chie}
For any $n\in\mathbb{Z}^2$, we have
$\kappa_{1}^2<|\chi^{(n)}|<\kappa_{2}^2$.
\end{prop}

\begin{proof}
Recalling \eqref{chi} and \eqref{beta}, we consider three cases:

\begin{enumerate}

\item[(i)] For $n\in U_{1}$, $\beta_{1}^{(n)}=(\kappa_{1}^2-|\boldsymbol{\alpha}^{(n)}|^2)^{1/2}$
and $\beta_{2}^{(n)}=(\kappa_{2}^2-|\boldsymbol{\alpha}^{(n)}|^2)^{1/2}$.
We have
\[
\chi^{(n)}=|\boldsymbol{\alpha}^{(n)}|^2+\beta_{1}^{(n)}\beta_{2}^{(n)}
=|\boldsymbol{\alpha}^{(n)}|^2+(\kappa_{1}^2-|\boldsymbol{\alpha}^{(n)}|^2)^{1/2
}(\kappa_{2}^2-|\boldsymbol{\alpha}^{(n)}|^2)^{1/2}.
\]
Consider the function
\[
g_1(t)=t+(k_1-t)^{1/2}(k_2-t)^{1/2}, \quad 0<k_1<k_2.
\]
It is easy to know that $g_1$ is decreasing for $0<t<k_1$. Hence
\[
k_1=g_1(k_1)<g_1(t)<g_1(0)=(k_1 k_2)^{1/2},
\]
which gives $\kappa_{1}^2<\chi^{(n)}<\kappa_{1}\kappa_{2}$.

\item[(ii)] For $n\in U_{2}\setminus U_{1}$, $\beta_{1}^{(n)} ={\rm
i}(|\boldsymbol{\alpha}^{(n)}|^2-\kappa_{1}^2)^{1/2}, \beta_{2}^{(n)}
=(\kappa_{2}^2-|\boldsymbol{\alpha}^{(n)}|^2)^{1/2}$.
We have
\[
\chi^{(n)}=|\boldsymbol{\alpha}^{(n)}|^2+{\rm
i}(|\boldsymbol{\alpha}^{(n)}|^2-\kappa_{1}^2)^{1/2}
(\kappa_{2}^2-|\boldsymbol{\alpha}^{(n)}|^2)^{1/2}
\]
and
\[
|\chi^{(n)}|^2=(\kappa_{1}^2+\kappa_{2}^2)|\boldsymbol{\alpha}^{(n)}|^2-(\kappa_{1}\kappa_{2})^2,
\]
which gives $\kappa_{1}^2<|\chi^{(n)}|<\kappa_{2}^2$.

\item[(iii)] For $n\notin U_{2}$,
$\beta_{1}^{(n)}={\rm
i}(|\boldsymbol{\alpha}^{(n)}|^2-\kappa_{1}^2)^{1/2},
 \beta_{2}^{(n)}={\rm
i}(|\boldsymbol{\alpha}^{(n)}|^2-\kappa_{2}^2)^{1/2}$.
We have
\[
\chi^{(n)}=|\boldsymbol{\alpha}^{(n)}|^2
-(|\boldsymbol{\alpha}^{(n)}|^2-\kappa_{1}^2)^{1/2}(|\boldsymbol{\alpha}^{(n)}|^2-\kappa_{2}^2)^{1/2}.
\]
Let
\[
g_2(t)=t-(t-k_1)^{1/2}(t-k_2)^{1/2}, \quad 0<k_1<k_2.
\]
It is easy to verify that the function $g_2$ is decreasing for $t>k_2$. Hence we
have
\[
(k_1+k_2)/2=g_2(\infty)<g_2(t)<g_2(k_2)=k_2,
\]
which gives
$(\kappa_{1}^2+\kappa_{2}^2)/2<\chi^{(n)}<\kappa_{2}^2$.
\end{enumerate}
Combining the above estimates, we get $\kappa_{1}^2<|\chi^{(n)}|<
\kappa_{2}^2$ for any $n\in\mathbb{Z}^2$.
\end{proof}

\begin{prop}\label{eg3}
The function $g_3(t)=t^k/ e^{(t^2-s^2)^{1/2}}$ satisfies $g_5(t)\leq
(s^2+k^2)^{k/2}$ for any $t>s>0, ~ k\in \mathbb{R}^1$.
\end{prop}

\begin{proof}
Using the change of variables $\tau=(t^2-s^2)^{1/2}$, we have
\[
 \hat{g}_3(\tau)=\frac{(\tau^2+s^2)^{k/2}}{e^{\tau}}.
\]
Taking the derivative of $\hat{g}_4$ gives
\[
 \hat{g}'_3(\tau)=-\frac{(\tau^2-k\tau+s^2)(\tau^2+s^2)^{\frac{k}{2}-1}}{e^{\tau}}.
\]
\begin{enumerate}

\item[(i)] If $s\geq k/2$, then $\hat{g}_3'\leq 0$ for $\tau>0$. The
function $\hat{g}_3$ is decreasing and reaches its maximum at $\tau=0$, i.e.,
\[
 g_3(t)\leq\hat{g}_3(0)=s^k.
\]

\item[(ii)] If $s<k/2$, then
$\hat{g}'_3<0$ for $\tau\in(0, (k-(k^2-4s^2)^{1/2})/2)\cup
((k+(k^2-4s^2)^{1/2})/2, \infty)$
and $\hat{g}_3>0$ for $\tau\in ((k-(k^2-4s^2)^{1/2})/2,
(k+(k^2-4s^2)^{1/2})/2)$. Thus
$\hat{g}_3$ reaches its maximum at either $\tau_1=0$ or
$\tau_2=(k+(k^2-4s^2)^{1/2})/2$. Thus we have
\[
g_3(t)=\hat{g}_3(\tau)\leq \max\{\hat{g}_3(\tau_1),\, \hat{g}_3(\tau_2)\}\leq
(s^2+k^2)^{k/2}.
\]
\end{enumerate}
The proof is completed by combining the above estimates.
\end{proof}

\begin{prop}\label{mfe}
For any $n\in\mathbb{Z}^2$, we have $\|M^{(n)}-\hat{M}^{(n)}\|_F\leq
\hat{K}$, where $\hat{K}=11\mu K/\kappa_{1}^4$.
\end{prop}

\begin{proof}
We consider the three cases:
\begin{enumerate}

\item[(i)] For $n\in U_{1}$, we have $|\boldsymbol{\alpha}^{(n)}|<\kappa_{1},
\beta_{1}^{(n)}=\Delta_{1}^{(n)}<\kappa_{1},
\beta_{2}^{(n)}=\Delta_{2}^{(n)}<\kappa_{2}$, and
$\Delta_{1}^{(n)}<\Delta_{2}^{(n)}$.
Using the facts that $\kappa_{1}<\kappa_{2}$, $\Delta_{i}^{(n)}\geq
{\Delta}^-_{\,i}$ for $n\in U_{1}$, we obtain from \eqref{chi} and
Proposition \ref{chie} and \ref{eg3} that

\begin{align*}
  |\varepsilon^{(n)}|\leq\frac{2e^{-\Delta_{2}^{(n)}{\rm Im}\zeta}}
  {e^{\Delta_{2}^{(n)}{\rm Im}\zeta}-e^{-\Delta_{2}^{(n)}{\rm Im}\zeta}}\leq\frac{2}{e^{2\Delta_{2}^{(n)}{\rm Im}\zeta}-1}
  \leq\frac{2}{e^{\Delta_{1}^{(n)}{\rm Im}\zeta}-1}
  \leq\frac{2}{e^{{\Delta}^-_{\,1}{\rm Im}\zeta}-1},
\end{align*}

\begin{align*}
 |\theta^{(n)}| \leq\frac{(e^{-\Delta_{2}^{(n)}{\rm Im}\zeta}
 +e^{-\Delta_{1}^{(n)}{\rm Im}\zeta})^2}
 {(1-e^{-2\Delta_{1}^{(n)}{\rm Im}\zeta})(1-e^{-2\Delta_{2}^{(n)}{\rm Im}\zeta})}\leq\frac{4e^{-2{\Delta}^-_{\,1}{\rm
Im}\zeta}}{(1-e^{-2{\Delta}^-_{\,1}{\rm Im}\zeta})^2}\leq\frac{4}
{(e^{\frac{1}{2}{\Delta}^-_{\,1}{\rm Im}\zeta}-1)^2},
\end{align*}

\begin{align*}
 |\eta^{(n)}|\leq\frac{e^{-2\Delta_{2}^{(n)}{\rm Im}\zeta}
 +e^{-2\Delta_{1}^{(n)}{\rm Im}\zeta}} {(1-e^{-2\Delta_{1}^{(n)}{\rm
Im}\zeta})(1-e^{-2\Delta_{2}^{(n)}{\rm Im}\zeta})}
 \leq\frac{2e^{-2{\Delta}^-_{\,1}{\rm Im}\zeta}}
 {(1-e^{-2{\Delta}^-_{\,1}{\rm Im}\zeta})^2}
 \leq\frac{2} {(e^{\frac{1}{2}{\Delta}^-_{\,1}{\rm Im}\zeta}-1)^2},
\end{align*}

\begin{align*}
 |\gamma^{(n)}|
 \leq\frac{e^{-2\Delta_{1}^{(n)}{\rm Im}\zeta}
 +e^{-4\Delta_{2}^{(n)}{\rm Im}\zeta}}
 {(1-e^{-2\Delta_{1}^{(n)}{\rm Im}\zeta})(1-e^{-2\Delta_{2}^{(n)}{\rm
Im}\zeta})^2}\leq\frac{2e^{-2\Delta_{1}^{(n)}{\rm
Im}\zeta}}{(1-e^{-2\Delta_{1}^{(n)}{\rm Im}\zeta})^3}
 \leq\frac{2}{(e^{\frac{1}{3}{\Delta}^-_{\,1}{\rm
Im}\zeta}-1)^3},
\end{align*}

\begin{align*}
  |\theta^{(n)}(\varepsilon^{(n)}+1)|
  \leq& \frac{4e^{-2\Delta_{1}^{(n)}{\rm Im}\zeta}}
 {(1-e^{-2\Delta_{1}^{(n)}{\rm Im}\zeta})^2}
 \frac{e^{\Delta_{2}^{(n)}{\rm Im}\zeta}+e^{-\Delta_{2}^{(n)}{\rm
Im}\zeta}}{e^{\Delta_{2}^{(n)}{\rm Im}\zeta}-e^{-\Delta_{2}^{(n)}{\rm Im}\zeta}}\\
  \leq&\frac{4e^{-2\Delta_{1}^{(n)}{\rm Im}\zeta}}
 {(1-e^{-2\Delta_{1}^{(n)}{\rm Im}\zeta})^2}
 \frac{2}{1-e^{-2\Delta_{1}^{(n)}{\rm
Im}\zeta}}\leq\frac{8}{(e^{\frac{1}{3}{\Delta}^-_{\,1}{\rm
Im}\zeta}-1)^3},
\end{align*}

\begin{align*}
 |\hat{\chi}^{(n)}-\chi^{(n)}| \leq 4 \kappa_{2}^2|\theta^{(n)}|\leq
F,
\end{align*}

\begin{align*}
 \max \Big\{&|\big((\alpha_1^{(n)})^2(\beta_{1}^{(n)}-\beta_{2}^{(n)})
 +\beta_{2}^{(n)}\chi^{(n)}\big)\chi^{(n)}\varepsilon^{(n)}|,
 |\alpha_1^{(n)}\alpha_{2}^{(n)}(\beta_{1}^{(n)}-\beta_{2}^{(n)})\chi^{(n)} \varepsilon^{(n)}|,\\
 &|\big((\alpha_2^{(n)})^2(\beta_{1}^{(n)}-\beta_{2}^{(n)})
 +\beta_{2}^{(n)}\chi^{(n)}\big)\chi^{(n)}\varepsilon^{(n)}|,
 |\beta_{2}^{(n)}\kappa_{2}^2\chi^{(n)}\varepsilon^{(n)}|\Big\}
 \leq 3 \kappa_{2}^5|\varepsilon^{(n)}|\leq F,
\end{align*}

\begin{align*}
\max\Big\{&|\big((\alpha_1^{(n)})^2(\beta_{1}^{(n)}-\beta_{2}^{(n)})
 +\beta_{2}^{(n)}\chi^{(n)}\big)(\hat{\chi}^{(n)}-\chi^{(n)})|,
 |\alpha_1^{(n)}\alpha_2^{(n)}(\beta_{1}^{(n)}-\beta_{2}^{(n)})(\hat{\chi}^{(n)}-\chi^{(n)})|,\\
 &|\big((\alpha_2^{(n)})^2(\beta_{1}^{(n)}-\beta_{2}^{(n)})
 +\beta_{2}^{(n)}\chi^{(n)}\big)(\hat{\chi}^{(n)}-\chi^{(n)})|,
 |\alpha_1^{(n)}\beta_{2}^{(n)}(\beta_{1}^{(n)}
 -\beta_{2}^{(n)})(\hat{\chi}^{(n)}-\chi^{(n)})|,\\
 &|\alpha_2^{(n)}\beta_{2}^{(n)}(\beta_{1}^{(n)}
 -\beta_{2}^{(n)})(\hat{\chi}^{(n)}-\chi^{(n)})|,
 |\beta_{2}^{(n)}\kappa_{2}^2(\hat{\chi}^{(n)}-\chi^{(n)})|\Big\}
 \leq 12 \kappa_{2}^5|\theta^{(n)}|\leq F,
\end{align*}

\begin{align*}
\max\Big\{&|4(\alpha_2^{(n)})^2\beta_{1}^{(n)}(\beta_{2}^{(n)})^2\theta^{(n)} (\varepsilon^{(n)}+1)|,
|4\alpha_1^{(n)}\alpha_2^{(n)}\beta_{1}^{(n)}(\beta_{2}^{(n)})^2\theta^{(n)} (\varepsilon^{(n)}+1)|,\\
&|4(\alpha_1^{(n)})^2\beta_{1}^{(n)}(\beta_{2}^{(n)})^2\theta^{(n)}(\varepsilon^{(n)}+1)|\Big\}\leq
4\kappa_{2}^5|\theta^{(n)}(\varepsilon^{(n)}+1)|\leq F,
\end{align*}

\begin{align*}
\max\Big\{&|2(\alpha_1^{(n)})^2\beta_{1}^{(n)}\kappa_{2}^2\eta^{(n)}|,
|2\alpha_1^{(n)}\alpha_2^{(n)}\beta_{1}^{(n)} \chi^{(n)}\eta^{(n)}|,\\
&2(\alpha_2^{(n)})^2\beta_{1}^{(n)}\kappa_{2}^2\eta^{(n)}|,
|2\beta_{1}^{(n)}(\beta_{2}^{(n)})^2 \kappa_{2}^2\eta^{(n)}|\Big\}
  \leq 2\kappa_{2}^5|\eta^{(n)}|\leq F,
\end{align*}

\begin{align*}
|2\alpha_1^{(n)}\alpha_1^{(n)}\beta_{1}^{(n)}\beta_{2}^{(n)}(\beta_{1}^{(n)}
-\beta_{2}^{(n)})\gamma^{(n)}|\leq
4\kappa_{2}^5|\gamma^{(n)}|\leq F,
\end{align*}

\begin{align*}
  \max \Big\{|2\alpha_1^{(n)}\beta_{1}^{(n)}\beta_{2}^{(n)}
  ( \kappa_{2}^2-2(\beta_{2}^{(n)})^2)\theta^{(n)}|,
  |2\alpha_2^{(n)}\beta_{1}^{(n)}\beta_{2}^{(n)}
  ( \kappa_{2}^2-2(\beta_{2}^{(n)})^2)\theta^{(n)}|\Big\}
  \leq 6 \kappa_{2}^5|\theta^{(n)}|\leq F,
\end{align*}

\item[(ii)] For $n\in U_{2}\setminus U_{1}$, we have
$|\boldsymbol{\alpha}^{(n)}|<\kappa_{2}, \beta_{1}^{(n)}={\rm
i}\Delta_{1}^{(n)}, \beta_{2}^{(n)}=\Delta_{2}^{(n)}<\kappa_{2}$,
$\Delta_{1}^{(n)}<( \kappa_{2}^2-\kappa_{1}^2)^{1/2}<\kappa_{2}$.
Using the facts that
$\Delta_{1}^{(n)}\geq {\Delta}^+_{1}, \Delta_{2}^{(n)}\geq
{\Delta}^-_{\,2}$ for $n\in U_{2}\setminus U_{1}$, we get
from Proposition \ref{chie} and \ref{eg3} that
\begin{align*}
  |\varepsilon^{(n)}|=\frac{2e^{-\Delta_{2}^{(n)}{\rm Im}\zeta}}
  {e^{\Delta_{2}^{(n)}{\rm Im}\zeta}-e^{-\Delta_{2}^{(n)}{\rm Im}\zeta}}
  \leq\frac{2}{e^{2\Delta_{2}^{(n)}{\rm Im}\zeta}-1}
  \leq\frac{2}{e^{2{\Delta}^-_{\,2}{\rm Im}\zeta}-1}
  \leq\frac{2}{e^{{\Delta}^-_{\,2}{\rm Im}\zeta}-1},
\end{align*}

\begin{align*}
 |\theta^{(n)}|\leq\frac{(e^{-\Delta_{2}^{(n)}{\rm Im}\zeta}
 +e^{-\Delta_{1}^{(n)}{\rm Re}\zeta})^2}
 {(1-e^{-2\Delta_{1}^{(n)}{\rm Re}\zeta})(1-e^{-2\Delta_{2,
j}^{(n)}{\rm Im}\zeta})}\leq\frac{(e^{-{\Delta}^-_{\,2}{\rm Im}\zeta}
 +e^{-{\Delta}^+_{1}{\rm Re}\zeta})^2}
 {(1-e^{-2{\Delta}^+_{1}{\rm Re}\zeta})
 (1-e^{-2{\Delta}^-_{\,2}{\rm Im}\zeta})},
\end{align*}

\begin{align*}
 |\eta^{(n)}|
 \leq&\frac{e^{-2\Delta_{2}^{(n)}{\rm Im}\zeta}
 +e^{-2\Delta_{1}^{(n)}{\rm Re}\zeta}}
 {(1-e^{-2\Delta_{1}^{(n)}{\rm Re}\zeta})(1-e^{-2\Delta_{2}^{(n)}{\rm
Im}\zeta})}\leq\frac{e^{-2{\Delta}^-_{2}{\rm Im}\zeta}
 +e^{-2{\Delta}^+_{1}{\rm Re}\zeta}}
 {(1-e^{-2{\Delta}^+_{1}{\rm Re}\zeta})
 (1-e^{-2{\Delta}^-_{2}{\rm Im}\zeta})}\\
 \leq&\frac{(e^{-{\Delta}^-_{2}{\rm Im}\zeta}
 +e^{-{\Delta}^+_{1}{\rm Re}\zeta})^2}
 {(1-e^{-2{\Delta}^+_{1}{\rm Re}\zeta})
 (1-e^{-2{\Delta}^-_{2}{\rm Im}\zeta})},
\end{align*}

\begin{align*}
 |\gamma^{(n)}|
 \leq&\frac{e^{-2\Delta_{1}^{(n)}{\rm Re}\zeta}
 +e^{-4\Delta_{2}^{(n)}{\rm Im}\zeta}}
 {(1-e^{-2\Delta_{1}^{(n)}{\rm Re}\zeta})(1-e^{-2\Delta_{2}^{(n)}{\rm
Im}\zeta})^2}\leq\frac{e^{-2{\Delta}^+_{1}{\rm Re}\zeta}
 +e^{-4{\Delta}^-_{\,2}{\rm Im}\zeta}}
 {(1-e^{-2{\Delta}^+_{1}{\rm Re}\zeta})
 (1-e^{-2{\Delta}^-_{\,2}{\rm Im}\zeta})^2}\\
 \leq&\frac{(e^{-{\Delta}^-_{\,2}{\rm Im}\zeta}
 +e^{-{\Delta}^+_{1}{\rm Re}\zeta})^2}
 {(1-e^{-2{\Delta}^+_{1}{\rm Re}\zeta})
 (1-e^{-2{\Delta}^-_{2}{\rm Im}\zeta})},
\end{align*}

\begin{align*}
  |\theta^{(n)}(\varepsilon^{(n)}+1)|
  \leq&\frac{(e^{-\Delta_{2}^{(n)}{\rm Im}\zeta}
 +e^{-\Delta_{1}^{(n)}{\rm Re}\zeta})^2}
 {(1-e^{-2\Delta_{1}^{(n)}{\rm Re}\zeta})(1-e^{-2\Delta_{2}^{(n)}{\rm Im}\zeta})}\frac{2}{1-e^{-2\Delta_{2}^{(n)}{\rm
Im}\zeta}}\\
  \leq&\frac{2(e^{-{\Delta}^-_{2}{\rm Im}\zeta}
 +e^{-{\Delta}^+_{1}{\rm Re}\zeta})^2}
 {(1-e^{-2{\Delta}^+_{1}{\rm Re}\zeta})
 (1-e^{-2{\Delta}^-_{2}{\rm Im}\zeta})^2},
\end{align*}

\begin{align*}
|\hat{\chi}^{(n)}-\chi^{(n)}| \leq\frac{4 \kappa_{2}^4}{\kappa_{1}^2}|\theta^{(n)}|\leq F,
\end{align*}

\begin{align*}
 \max\Big\{&|\big((\alpha_1^{(n)})^2(\beta_{1}^{(n)}-\beta_{2}^{(n)})
 +\beta_{2}^{(n)}\chi^{(n)}\big)\chi^{(n)}\varepsilon^{(n)}|,
 |\alpha_1^{(n)}\alpha_2^{(n)}(\beta_{1}^{(n)}-\beta_{2}^{(n)})\chi^{(n)} \varepsilon ^{(n)}|,\\
 &|\big((\alpha_2^{(n)})^2(\beta_{1}^{(n)}-\beta_{2}^{(n)})
 +\beta_{2}^{(n)}\chi^{(n)}\big)\chi^{(n)}\varepsilon^{(n)}|,
 |\beta_{2}^{(n)} \kappa_{2}^2\chi^{(n)}\varepsilon^{(n)}|\Big\}
 \leq 3 \kappa_{2}^5|\varepsilon^{(n)}|\leq F,
\end{align*}

\begin{align*}
 \max\Big\{&|\big((\alpha_1^{(n)})^2(\beta_{1}^{(n)}-\beta_{2}^{(n)})
 +\beta_{2}^{(n)}\chi^{(n)}\big)(\hat{\chi}^{(n)}-\chi^{(n)})|,
 |\alpha_1^{(n)}\alpha_2^{(n)}(\beta_{1}^{(n)}-\beta_{2,
j}^{(n)})(\hat{\chi}^{(n)}-\chi^{(n)})|,\\
 &|\big((\alpha_2^{(n)})^2(\beta_{1}^{(n)}-\beta_{2}^{(n)})
 +\beta_{2}^{(n)}\chi^{(n)}\big)(\hat{\chi}^{(n)}-\chi^{(n)})|,
 |\alpha_1^{(n)}\beta_{2}^{(n)}(\beta_{1}^{(n)}
 -\beta_{2}^{(n)})(\hat{\chi}^{(n)}-\chi^{(n)})|,\\
 &|\alpha_2^{(n)}\beta_{2}^{(n)}(\beta_{1}^{(n)}
 -\beta_{2}^{(n)})(\hat{\chi}^{(n)}-\chi^{(n)})|,
 |\beta_{2}^{(n)} \kappa_{2}^2(\hat{\chi}^{(n)}-\chi^{(n)})|\Big\}
 \leq\frac{12 \kappa_{2}^7}{ \kappa_{1}^2}|\theta^{(n)}|
  \leq F,
\end{align*}

\begin{align*}
\max\Big\{&|4(\alpha_2^{(n)})^2\beta_{1}^{(n)}(\beta_{2}^{(n)})^2\theta^{(n)} (\varepsilon^{(n)}+1)|,
|4\alpha_1^{(n)}\alpha_2^{(n)}\beta_{1}^{(n)}(\beta_{2}^{(n)})^2
\theta^{(n)} (\varepsilon^{(n)}+1)|,\\
  &|4 (\alpha_1^{(n)})^2\beta_{1}^{(n)}
  (\beta_{2}^{(n)})^2\theta^{(n)}(\varepsilon^{(n)}+1)|\Big\}
  \leq4 \kappa_{2}^5|\theta^{(n)}(\varepsilon^{(n)}+1)|\leq F,
\end{align*}

\begin{align*}
  \max\Big\{&|2
(\alpha_1^{(n)})^2\beta_{1}^{(n)} \kappa_{2}^2\eta^{(n)}|,
|2\alpha_1^{(n)}\alpha_2^{(n)}\beta_{1}^{(n)}\chi^{(n)}\eta^{(n)}|,\\
&|2 (\alpha_2^{(n)})^2\beta_{1}^{(n)}\kappa_{2}^2\eta^{(n)}|,
|2\beta_{1}^{(n)}(\beta_{2}^{(n)})^2 \kappa_{2}^2\eta^{(n)}|\Big\}
  \leq2 \kappa_{2}^5|\eta^{(n)}|\leq F,
\end{align*}

\begin{align*}
|2\alpha_1^{(n)}\alpha_2^{(n)}\beta_{1}^{(n)}\beta_{2}^{(n)}(\beta_{1}^{(n)}
-\beta_{2}^{(n)})\gamma^{(n)}|
\leq4 \kappa_{2}^5|\gamma^{(n)}|\leq F,
\end{align*}

\begin{align*}
  \max\Big\{&|2\alpha_1^{(n)}\beta_{1}^{(n)}\beta_{2}^{(n)}
  ( \kappa_{2}^2-2(\beta_{2}^{(n)})^2)\theta^{(n)}|,
  |2\alpha_2^{(n)}\beta_{1}^{(n)}\beta_{2}^{(n)}
  ( \kappa_{2}^2-2(\beta_{2}^{(n)})^2)\theta^{(n)}|\Big\}\\
  &\leq 6 \kappa_{2}^5|\theta^{(n)}|\leq F.
\end{align*}

\item[(iii)] For $n\notin U_{2}$, we have
$\kappa_{2}<|\boldsymbol{\alpha}^{(n)}|, \beta_{1}^{(n)}={\rm i}\Delta_{1}^{(n)}, \beta_{2}^{(n)}={\rm
i}\Delta_{2}^{(n)}$,
and $\Delta_{2}^{(n)}<\Delta_{1}^{(n)}<|\boldsymbol{\alpha}^{(n)}|$.
Noting ${\rm Re}\zeta\geq 1$, we obtain from Proposition \ref{eg3} that

\begin{align*}
  |\varepsilon^{(n)}|\leq\frac{2e^{-\Delta_{2}^{(n)}{\rm Re}\zeta}}
  {e^{\Delta_{2}^{(n)}{\rm Re}\zeta}-e^{-\Delta_{2}^{(n)}{\rm Re}\zeta}}
  \leq\frac{2}{e^{\Delta_{2}^{(n)}{\rm Re}\zeta}}
  \frac{1}{e^{\Delta_{2}^{(n)}{\rm Re}\zeta}-1}
\leq\frac{2}{e^{(|\boldsymbol{\alpha}^{(n)}|^2-\kappa_{2}^2)^{1/2}
}}\frac{1}{e^{{\Delta}^+_{2}{\rm Re}\zeta}-1},
\end{align*}

\begin{align*}
 |\theta^{(n)}|\leq&\frac{(e^{-\Delta_{2}^{(n)}{\rm Re}\zeta}
 +e^{-\Delta_{1}^{(n)}{\rm Re}\zeta})^2}{(1-e^{-2\Delta_{1}^{(n)}{\rm
Re}\zeta})(1-e^{-2\Delta_{2}^{(n)}{\rm Re}\zeta})}
 \leq\frac{4e^{-2\Delta_{2}^{(n)}{\rm Re}\zeta}}{(1-e^{-2\Delta_{2}^{(n)}{\rm Re}\zeta})^2}\\
 \leq&\frac{4}{e^{\Delta_{2}^{(n)}{\rm Re}\zeta}}
 \frac{e^{-\Delta_{2}^{(n)}{\rm
Re}\zeta}}{(1-e^{-2\Delta_{2}^{(n)}{\rm Re}\zeta})^2}
 \leq\frac{4}{e^{\Delta_{2}^{(n)}}}\frac{1}{(e^{\frac{1}{2}\Delta_{2}^{(n)}{\rm
Re}\zeta}-1)^2}\\
 \leq&\frac{4}{e^{(|\boldsymbol{\alpha}^{(n)}|^2-\kappa_{2}^2)^{1/2}}}
 \frac{1}{(e^{\frac{1}{2}{\Delta}^+_{2}{\rm Re}\zeta}-1)^2},
\end{align*}

\begin{align*}
 |\eta^{(n)}|\leq&\frac{e^{-2\Delta_{2}^{(n)}{\rm Re}\zeta}
 +e^{-2\Delta_{1}^{(n)}{\rm Re}\zeta}}
 {(1-e^{-2\Delta_{1}^{(n)}{\rm Re}\zeta})(1-e^{-2\Delta_{2}^{(n)}{\rm
Re}\zeta})}\leq\frac{2e^{-2\Delta_{2}^{(n)}{\rm
Re}\zeta}}{(1-e^{-2\Delta_{2}^{(n)}{\rm Re}\zeta})^2}\\
 \leq&\frac{2}{e^{\Delta_{2}^{(n)}{\rm Re}\zeta}}
 \frac{e^{-\Delta_{2}^{(n)}{\rm Re}\zeta}}{(1-e^{-2\Delta_{2}^{(n)}{\rm Re}\zeta})^2}
 \leq\frac{2}{e^{\Delta_{2}^{(n)}}}
 \frac{1}{(e^{\frac{1}{2}\Delta_{2}^{(n)}{\rm Re}\zeta}-1)^2}\\
 \leq&\frac{2}{e^{(|\boldsymbol{\alpha}^{(n)}|^2-\kappa_{2}^2)^{1/2}}}
 \frac{1}{(e^{\frac{1}{2}{\Delta}^+_{2}{\rm Re}\zeta}-1)^2},
\end{align*}

\begin{align*}
 |\gamma^{(n)}|\leq&\frac{e^{-2\Delta_{1}^{(n)}{\rm Re}\zeta}
 +e^{-4\Delta_{2}^{(n)}{\rm Re}\zeta}}{(1-e^{-2\Delta_{1}^{(n)}{\rm
Re}\zeta})(1-e^{-2\Delta_{2}^{(n)}{\rm Re}\zeta})^2}
 \leq\frac{2e^{-2\Delta_{2}^{(n)}{\rm
Re}\zeta}}{(1-e^{-2\Delta_{2}^{(n)}{\rm Re}\zeta})^3}\\
 \leq&\frac{2}{e^{\Delta_{2}^{(n)}{\rm Re}\zeta}}
 \frac{e^{-\Delta_{2}^{(n)}{\rm
Re}\zeta}}{(1-e^{-2\Delta_{2}^{(n)}{\rm Re}\zeta})^3}
 \leq\frac{2}{e^{\Delta_{2}^{(n)}}}
 \frac{1}{(e^{\frac{1}{3}\Delta_{2}^{(n)}{\rm Re}\zeta}-1)^3}\\
 \leq&\frac{2}{e^{(|\boldsymbol{\alpha}^{(n)}|^2-\kappa_{2}^2)^{1/2}}}
 \frac{1}{(e^{\frac{1}{3}{\Delta}^+_{2}{\rm Re}\zeta}-1)^3},
\end{align*}

\begin{align*}
  |\theta^{(n)}(\varepsilon^{(n)}+1)|
  \leq&\frac{4e^{-2\Delta_{2}^{(n)}{\rm
Re}\zeta}}{(1-e^{-2\Delta_{2}^{(n)}{\rm Re}\zeta})^2}
  \frac{2}{1-e^{-2\Delta_{2}^{(n)}{\rm Re}\zeta}}
  \leq\frac{8}{e^{\Delta_{2}^{(n)}{\rm Re}\zeta}}
  \frac{e^{-\Delta_{2}^{(n)}{\rm
Re}\zeta}}{(1-e^{-2\Delta_{2}^{(n)}{\rm Re}\zeta})^3}\\
  \leq&\frac{8}{e^{\Delta_{2}^{(n)}}}
  \frac{1}{(e^{\frac{1}{3}\Delta_{2}^{(n)}{\rm Re}\zeta}-1)^3}
\leq\frac{8}{e^{(|\boldsymbol{\alpha}^{(n)}|^2-\kappa_{2}^2)^{1/2}
}}\frac{1}{(e^{\frac{1}{3}{\Delta}^+_{2}{\rm
Re}\zeta}-1)^3},
\end{align*}

\begin{align*}
|\hat{\chi}^{(n)}-\chi^{(n)}|\leq&
\frac{4|\boldsymbol{\alpha}^{(n)}|^4|\theta^{(n)}|}{\kappa_{1}^2}
\leq\frac{16}{\kappa_{1}^2}\frac{|\boldsymbol{\alpha}^{(n)}|^4}{e^{(|\boldsymbol{
\alpha}^{(n)}|^2-\kappa_{2}^2)^{1/2}}} \frac{1}{(e^{\frac{1}{2}{\Delta}^+_{2}{\rm
Re}\zeta}-1)^2}\\
\leq&\frac{16( \kappa_{2}^2+16)^2}{\kappa_{1}^2(e^{\frac{1}{2}{\Delta}^+_{2}{\rm
Re}\zeta}-1)^2}\leq F,
\end{align*}

\begin{align*}
 \max\Big\{&|\big((\alpha_1^{(n)})^2(\beta_{1}^{(n)}-\beta_{2}^{(n)})
 +\beta_{2}^{(n)}\chi^{(n)}\big)\chi^{(n)}\varepsilon^{(n)}|,
 |\alpha_1^{(n)}\alpha_2^{(n)}(\beta_{1}^{(n)}-\beta_{2}^{(n)})\chi^{(n)} \varepsilon^{(n)}|,\\
 &|\big((\alpha_2^{(n)})^2(\beta_{1}^{(n)}-\beta_{2}^{(n)})
 +\beta_{2}^{(n)}\chi^{(n)})\chi^{(n)}\varepsilon^{(n)}|,
 |\beta_{2}^{(n)} \kappa_{2}^2\chi^{(n)}\varepsilon^{(n)}|\Big\}\\
 &\leq 3 \kappa_{2}^2|\boldsymbol{\alpha}^{(n)}|^3|\varepsilon^{(n)}|
 \leq\frac{|\boldsymbol{\alpha}^{(n)}|^3}{e^{(|\boldsymbol{\alpha}
^{(n)}|^2- \kappa_{2}^2)^{1/2}}}\frac{6 \kappa_{2}^2}{e^{{\Delta}^+_{2}{\rm Re}\zeta}-1}\leq\frac{6 \kappa_{2}^2(
\kappa_{2}^2+9)^{3/2}}{e^{{\Delta}^+_{2}{\rm Re}\zeta}-1}\leq
F,
\end{align*}

\begin{align*}
 \max\Big\{&|\big( (\alpha_1^{(n)})^2(\beta_{1}^{(n)}-\beta_{2}^{(n)})
 +\beta_{2}^{(n)}\chi^{(n)}\big)(\hat{\chi}^{(n)}-\chi^{(n)})|,
 |\alpha_1^{(n)}\alpha_2^{(n)}(\beta_{1}^{(n)}-\beta_{2}^{(n)})(\hat{\chi}^{(n)}-\chi^{(n)})|,\\
 &|\big((\alpha_2^{(n)})^2(\beta_{1}^{(n)}-\beta_{2}^{(n)})
 +\beta_{2}^{(n)}\chi^{(n)})(\hat{\chi}^{(n)}-\chi^{(n)})|,
 |\alpha_1^{(n)}\beta_{2}^{(n)}(\beta_{1}^{(n)}
 -\beta_{2}^{(n)})(\hat{\chi}^{(n)}-\chi^{(n)})|,\\
 &|\alpha_2^{(n)}\beta_{2}^{(n)}(\beta_{1}^{(n)}
 -\beta_{2}^{(n)})(\hat{\chi}^{(n)}-\chi^{(n)})|,
 |\beta_{2}^{(n)} \kappa_{2}^2(\hat{\chi}^{(n)}-\chi^{(n)})|\Big\}\\
 &\leq\frac{12|\boldsymbol{\alpha}^{(n)}|^7|\theta^{(n)}|}{ \kappa_{1}^2}
\leq\frac{48}{ \kappa_{1}^2}\frac{|\boldsymbol{\alpha}^{(n)}|^7}{e^{
(|\boldsymbol{\alpha}^{(n)}|^2-\kappa_{2}^2)^{1/2}}}
\frac{1}{(e^{\frac{1}{2}{\Delta}^+_{2}{\rm Re}\zeta}-1)^2}\\
&\leq\frac{48( \kappa_{2}^2+49)^{7/2}}{ \kappa_{1}^2
(e^{\frac{1}{2}{\Delta}^+_{2}{\rm Re}\zeta}-1)^2}\leq F,
\end{align*}

\begin{align*}
\max\Big\{&|4(\alpha_2^{(n)})^2\beta_{1}^{(n)}(\beta_{2}^{(n)})^2\theta^{(n)} (\varepsilon^{(n)}+1)|,
|4\alpha_1^{(n)}\alpha_2^{(n)}\beta_{1}^{(n)}(\beta_{2}^{(n)})^2\theta^{(n)} (\varepsilon^{(n)}+1)|,\\
&|4 (\alpha_1^{(n)})^2\beta_{1}^{(n)}(\beta_{2}^{(n)})^2\theta^{(n)}(\varepsilon^{(n)}+1)|\Big\}\\
  &\leq 4|\boldsymbol{\alpha}^{(n)}|^5|\theta^{(n)}(\varepsilon^{(n)}+1)|
\leq\frac{32|\boldsymbol{\alpha}^{(n)}|^5}{e^{(|\boldsymbol{\alpha}^{(n)}
|^2-(\kappa_{2}^2)^{1/2}}}\frac{1}{(e^{\frac{1}{3}{\Delta}^+_{2}{\rm Re}\zeta}-1)^3}\leq\frac{32( \kappa_{2}^2+25)^{5/2}}
  {(e^{\frac{1}{3}{\Delta}^+_{2}{\rm Re}\zeta}-1)^3}\leq F,
\end{align*}

\begin{align*}
\max\Big\{&|2(\alpha_1^{(n)})^2\beta_{1}^{(n)} \kappa_{2}^2\eta^{(n)}|,
|2\alpha_1^{(n)}\alpha_2^{(n)}\beta_{1}^{(n)}\chi^{(n)}\eta^{(n)}|,|2(\alpha_2^{
(n)})^2 \beta_{1}^{(n)}\kappa_{2}^2\eta^{(n)}
|,|2\beta_{1}^{(n)}(\beta_{2}^{(n)})^2
\kappa_{2}^2\eta^{(n)}|\Big\} \\
&\leq 2 \kappa_{2}^2|\boldsymbol{\alpha}^{(n)}|^3|\eta^{(n)}|
  \leq4 \kappa_{2}^2\frac{|\boldsymbol{\alpha}^{(n)}|^3}
  {e^{(|\boldsymbol{\alpha}^{(n)}|^2-\kappa_{2}^2)^{1/2}}}
 \frac{1}{(e^{\frac{1}{2}{\Delta}^+_{2}{\rm Re}\zeta}-1)^2}
 \leq\frac{4 \kappa_{2}^2( \kappa_{2}^2+9)^{3/2}}
  {(e^{\frac{1}{2}{\Delta}^+_{2}{\rm Re}\zeta}-1)^2}
  \leq F,
\end{align*}

\begin{align*}
&|2\alpha_1^{(n)}\alpha_2^{(n)}\beta_{1}^{(n)}\beta_{2}^{(n)}(\beta_{1}^{(n)}
-\beta_{2}^{(n)})\gamma^{(n)}|
  \leq4|\boldsymbol{\alpha}^{(n)}|^5|\gamma^{(n)}|\\
\leq&\frac{8|\boldsymbol{\alpha}^{(n)}|^5}{e^{(|\boldsymbol{\alpha}^{(n)}
|^2- \kappa_{2}^2)^{1/2}}}\frac{1}{(e^{\frac{1}{3}{\Delta}^+_{2}{\rm
Re}\zeta}-1)^3} \leq\frac{8( \kappa_{2}^2+25)^{5/2}}
 {(e^{\frac{1}{3}{\Delta}^+_{2}{\rm Re}\zeta}-1)^3}\leq F,
\end{align*}

\begin{align*}
  \max\Big\{&|2\alpha_1^{(n)}\beta_{1}^{(n)}\beta_{2}^{(n)}
  (\kappa_{2}^2-2(\beta_{2}^{(n)})^2)\theta^{(n)}|,
  |2\alpha_2^{(n)}\beta_{1}^{(n)}\beta_{2}^{(n)}
  (\kappa_{2}^2-2(\beta_{2}^{(n)})^2)\theta^{(n)}|\Big\}\\
  \leq&6|\boldsymbol{\alpha}^{(n)}|^5|\theta^{(n)}|
\leq\frac{|\boldsymbol{\alpha}^{(n)}|^5}{e^{(|\boldsymbol{\alpha}^{(n)}
|^2-\kappa_{2}^2)^{1/2}}}\frac{24}{(e^{\frac{1}{2}{\Delta}^+_{2}{\rm
Re}\zeta}-1)^2}\leq\frac{24((\kappa_2)^2+25)^{\frac{5}{2}}}
 {(e^{\frac{1}{2}{\Delta}^+_{2}{\rm Re}\zeta}-1)^2}\leq
F,
\end{align*}
where we have used the estimate for $g_3$ and the facts that
$\Delta_{i}^{(n)}\geq \Delta_{i}^{+}$ for $n\notin U_{2}$.
\end{enumerate}

It follows from Proposition \ref{chie} and the estimate
$|\hat{\chi}^{(n)}-\chi^{(n)}|\leq K$ that
$\kappa_{1}^2-K\leq |\hat{\chi}^{(n)}|\leq \kappa_{2}^2+K$. 
Again, we may choose some proper PML parameters $\sigma$ and $\delta$ such
that $K\leq \kappa_{1}^2/2$, which gives $|\hat{\chi}^{(n)}|\geq
\kappa_{1}^2/2$. Using the matrix Frobenius norm and combining all the above
estimates, we get
\begin{align*}
\|&M^{(n)}-\hat{M}^{(n)}\|_F^2
 \leq\frac{4\mu^2}{\kappa_{1}^8}\Big(
|\big((\alpha_1^{(n)})^2(\beta_{1}^{(n)}-\beta_{2}^{(n)})
 +\beta_{2}^{(n)}\chi^{(n)}\big)\chi^{(n)}\varepsilon^{(n)}|^2
 +|2(\alpha_1^{(n)})^2\beta_{1}^{(n)}\kappa_{2}^2\eta^{(n)}|^2\\
 &+|\big((\alpha_1^{(n)})^2(\beta_{1}^{(n)}-\beta_{2}^{(n)})
 +\beta_{2}^{(n)}\chi^{(n)}\big)(\hat{\chi}^{(n)}-\chi^{(n)})|^2
 +|4(\alpha_2^{(n)})^2\beta_{1}^{(n)}(\beta_{2}^{(n)})^2\theta^{(n)}
(\varepsilon^{(n)}+1)|^2\\
 &+2|\alpha_1^{(n)}\alpha_2^{(n)}(\beta_{1}^{(n)}-\beta_{2}^{(n)})\chi^{(n)}
\varepsilon^{(n)}|^2+2|\alpha_1^{(n)}\alpha_2^{(n)}(\beta_{1}^{(n)}-\beta_{2}^{
(n)})(\hat{\chi}^{(n)}-\chi^{(n)})|^2\\
 &+2|2\alpha_1^{(n)}\alpha_2^{(n)}\beta_{1}^{(n)}\chi^{(n)}\eta^{(n)}|^2
+2|4\alpha_1^{(n)}\alpha_2^{(n)}\beta_{1}^{(n)}(\beta_{2}^{(n)})^2\theta^{(n)}
(\varepsilon^{(n)}+1)|^2
 +|2(\alpha_2^{(n)})^2\beta_{1}^{(n)} \kappa_{2}^2\eta^{(n)}|^2\\
 &+2|2\alpha_1^{(n)}\alpha_2^{(n)}\beta_{1}^{(n)}\beta_{2}^{(n)}(\beta_{1}^{(n)
}-\beta_{2}^{(n)})\gamma^{(n)}|^2
 +2|\alpha_1^{(n)}\beta_{2}^{(n)}(\beta_{1}^{(n)}
 -\beta_{2}^{(n)})(\hat{\chi}^{(n)}-\chi^{(n)})|^2\\
 &+2|2\alpha_1^{(n)}\beta_{1}^{(n)}\beta_{2}^{(n)}
  ( \kappa_{2}^2-2(\beta_{2}^{(n)})^2)\theta^{(n)}|^2
  +|\big((\alpha_2^{(n)})^2(\beta_{1}^{(n)}-\beta_{2}^{(n)})
 +\beta_{2}^{(n)}\chi^{(n)}\big)\chi^{(n)}\varepsilon^{(n)}|^2\\
 &+|\big((\alpha_2^{(n)})^2(\beta_{1}^{(n)}-\beta_{2}^{(n)})
 +\beta_{2}^{(n)}\chi^{(n)}\big)(\hat{\chi}^{(n)}-\chi^{(n)})|^2
+|4(\alpha_1^{(n)})^2\beta_{1}^{(n)}(\beta_{2}^{(n)})^2\theta^{(n)}
(\varepsilon^{(n)}+1)|^2\\
 & +2|\alpha_2^{(n)}\beta_{2}^{(n)}(\beta_{1}^{(n)}
 -\beta_{2}^{(n)})(\hat{\chi}^{(n)}-\chi^{(n)})|^2
 +2|2\alpha_2^{(n)}\beta_{1}^{(n)}\beta_{2}^{(n)}
  (\kappa_{2}^2-2(\beta_{2}^{(n)})^2)\theta^{(n)}|^2 \\
  &  +|\beta_{2}^{(n)}\kappa_{2}^2\chi^{(n)}\varepsilon^{(n)}|^2
  +|\beta_{2}^{(n)} \kappa_{2}^2(\hat{\chi}^{(n)}-\chi^{(n)})|^2
  +|2\beta_{1}^{(n)}(\beta_{2}^{(n)})^2 \kappa_{2}^2\eta^{(n)}|^2
\Big)\leq \frac{116\mu^2}{\kappa_{1}^8}K^2,
\end{align*}
which completes the proof.
\end{proof}

\end{document}